\documentclass[12pt]{amsart}
\usepackage{amssymb}
\usepackage{amsfonts}
\usepackage{amsmath,amsthm,amsfonts,amscd,amssymb,mathdots,amstext}
\usepackage{fullpage}
\usepackage[all]{xy}
 \usepackage[usenames,dvipsnames]{color}

\usepackage{graphicx}
\usepackage{bbding}
\usepackage[misc]{ifsym}

\usepackage[margin=3.65cm]{geometry}

\theoremstyle{plain}
\newtheorem{thm}{Theorem}[section]
\newtheorem{cor}[thm]{Corollary}
\newtheorem{prop}[thm]{Proposition}
\newtheorem{lem}[thm]{Lemma}

\newcounter{theoremintro}
\newtheorem{theoremi}[theoremintro]{Theorem}

\theoremstyle{definition}
\newtheorem{rem}[thm]{Remark}

\newtheorem{defn}[thm]{Definition}

\newtheorem{eg}[thm]{Example}

\newtheorem{problem}[thm]{Problem}

\newcommand{\bD}{{\mathbb{D}}}

\newcommand{\bT}{{\mathbb{T}}}

  \newcommand{\M}{{\mathcal{M}}}
  \newcommand{\N}{{\mathcal{N}}}

  \newcommand{\T}{{\mathcal{T}}}

\renewcommand{\phi}{\varphi}
\newcommand{\upchi}{{\raise.35ex\hbox{$\chi$}}}

\newcommand{\ran}{\operatorname{ran}}

\newcommand{\ol}{\overline}

\newcommand{\td}{\widetilde}

\newcommand{\Dbar}{\overline{\mathbb{D}}}

\usepackage[hidelinks]{hyperref}
\hypersetup{
  pdftitle={The invariant subspace problem and Rosenblum operators II},
  pdfauthor={Junsheng Fang, Bingzhe Hou, Chunlan Jiang, Yuanhang Zhang},
  pdfkeywords={Invariant subspace problem, Rosenblum operators, ideal property, Pearcy's problem, Halmos's third problem, conjugate bi-geometric operators}
}
\renewcommand{\H}{{\mathcal{H}}}
\begin{document}

\title[The invariant subspace problem and Rosenblum operators]{The invariant subspace problem and Rosenblum operators II}

\author{Junsheng Fang}
\address{Junsheng Fang, School of Mathematics, Hebei Normal University, 050016, Shijiazhuang, P. R. China}
\email{jfang@hebtu.edu.cn}

\author{Bingzhe Hou}
\address{Bingzhe Hou, School of Mathematics, Jilin University, 130012, Changchun, P. R. China}
\email{houbz@jlu.edu.cn}
	
\author{Chunlan Jiang}
\address{Chunlan Jiang, School of Mathematics, Hebei Normal University, 050016, Shijiazhuang, P. R. China}
\email{cljiang@hebtu.edu.cn}

\author{Yuanhang Zhang}
\address{Yuanhang Zhang, School of Mathematics, Jilin University, 130012, Changchun, P. R. China}
\email{zhangyuanhang@jlu.edu.cn}

\thanks{Junsheng Fang is supported in part by National Natural Science Foundation of China (No.: 12071109); Chunlan Jiang and Bingzhe Hou are supported in part by National Natural Science Foundation of China (No.:12471120); Chunlan Jiang and Junsheng
Fang are partly supported by the Hebei Natural Science Foundation (No.: A2023205045); Yuanhang Zhang is supported in part by National Natural Science Foundation of China (No.: 12471123).
}
	
\date{}
\thanks{}
\begin{abstract}
Let $\H$ be the separable, infinite-dimensional complex Hilbert space. In the first paper of this series, we introduced, by means of the Rosenblum operators, the shift representation operators $K_x=\sum_{n=0}^{\infty}T^nx\otimes e_n$, which relate an operator $T$ with $r(T)<1$ to the unilateral shift $S$ and reveal its hidden analytic structure. In this paper we replace the pair $(H^2,S)$ by the Sobolev disk algebra $R(\bD)$ and the multiplication operator $M_z$, whose canonical left inverse $B$ plays the role of $S^*$, and study the two-parameter family $K_{x,y}=\sum_{n=0}^{\infty}T^nx\otimes B^{*n}y$. We prove that $T\in B(R(\bD))$ with $r(T)<1$ is intransitive whenever there exist a nonzero vector $g$ and a nontrivial $M_z$-invariant manifold $\M$ such that $f(T)g$ has a zero in the closed unit disk for every $f\in\M$; the proof explicitly constructs a nontrivial invariant subspace of $T$. We further introduce the ideal property and show that every intransitive operator on $\H$ is unitarily equivalent to an operator on $R(\bD)$ with this property, so that the Invariant Subspace Problem for arbitrary operators reduces to a problem about a concrete function algebra. Applications include a characterization of intransitivity through the rationality of the generating function $\sum_{n=0}^{\infty}\langle T^n\xi,\eta\rangle z^n$ and new results on Pearcy's problem. Our main application is Halmos's third problem, which has remained open for more than fifty years: we reduce it to invertible operators with conjugate bi-geometric form, and we prove that $T^{-1}$ is intransitive for every invertible $T$ admitting a nontrivial projection $P$ with $PTP=TP$ and $PT(I-P)$ of rank one. Finally, we exhibit tridiagonal operators that are intransitive and admit infinite decreasing chains of invariant subspaces.
\end{abstract}

\date{\today}

\maketitle

Let $\H$ denote the unique (up to isomorphism) separable, infinite-dimensional complex Hilbert space. A {\it subspace} of $\H$ is a linear manifold which is closed in the norm topology. If $\mathcal{L}\subset \H$, then the {\it span} of $\mathcal{L}$, denoted by $\bigvee\mathcal{L}$, is the intersection of all subspaces containing $\mathcal{L}$. We say that the subspace $\M$ is invariant under the operator $T$ if $T\M\subset \M$. Denote by $B(\H)$ the set of all bounded linear operators acting on $\H$, and for $T\in B(\H)$ denote by $\textup{Lat}T$ the collection of all subspaces of $\H$ invariant under $T$; $T$ is called {\it intransitive} if it leaves invariant subspaces other than $\{0\}$ or $\H$; otherwise, it is called {\it transitive}, which means that $\textup{Lat}T=\{\{0\},\H\}$. Denote by $K(\H)$ the closed ideal of $B(\H)$ of all compact operators. Denote by $\ker T$ and $\ran T$ the kernel and the range of $T\in B(\H)$, and by $\sigma(T),\sigma_p(T),r(T)$ its spectrum, point spectrum and spectral radius.

Given $f,g\in\H$, define the rank one operator $f\otimes g$ on $\H$ by $(f\otimes g)h=\langle h,g\rangle f$, $h\in\H$. For $T\in B(\H)$, we write $\{T\}'=\{A\in B(\H):TA=AT\}$ for the {\it commutant} of $T$, a weakly closed subalgebra with identity. For $A,B\in B(\H)$, the {\bf Rosenblum operator} is defined by $\tau_{A,B}(X)=AX-XB$, $X\in B(\H)$. The operator $\tau_{A,B}$ was first studied systematically by Rosenblum \cite{Ros}; later, many scholars obtained abundant results, and we refer to \cite{DavRos74,Fia78,Fia79,BhaRos97} and the references therein for more details.

Let $\bT$ be the unit circle with the normalized Lebesgue measure $m$. For $p=2,\infty$, let $H^p$ be the classical Hardy space of all functions in $L^p(\bT)$ with analytic extensions to the open disk $\bD$. Let $\{e_n\}_{n=0}^\infty$ to denote the canonical orthonormal basis $\{z^n\}_{n=0}^\infty$ of $H^2$ and $S$ be the {\it unilateral shift of multiplicity $1$}, that is, the unique operator with $Se_n=e_{n+1}$ for $n=0,1,2,\cdots$. Up to Hilbert space isomorphism, $\H$, $H^2$ and $\ell^2$ are the same; there is no harm in letting $S$ also denote the unilateral shift in $B(\H)$ or $B(\ell^2)$.

The \textbf{Invariant Subspace Problem} is one of the best-known unsolved problems in functional analysis: does every bounded operator on $\H$ have a nontrivial closed invariant subspace? In the 1930s, von Neumann found a beautiful technique for producing invariant subspaces of compact operators, which Aronszajn and Smith \cite{AS} extended to Banach spaces and Lomonosov \cite{Lom} developed into the statement that every operator commuting with a nonzero compact operator has a nontrivial hyperinvariant subspace. The answer is negative for general Banach spaces, by the examples of Enflo \cite{En} and Read \cite{R1,R2}; for operators on $\H$ the problem remains open, although positive results are known for many classes, such as hyponormal operators with thick spectra \cite{Brown87}. We refer the reader to \cite{RR2} and \cite{CP} for comprehensive accounts, and to \cite{RR-tri} for the theory of simultaneous triangularization.

The unilateral shift $S$ has many remarkable properties; for instance, by the celebrated Beurling theorem \cite{Beu48}, every nontrivial invariant subspace of $S$ is of the form $\phi H^2$ for some inner function $\phi$. In the first paper of this series \cite{FHJZ1}, we related an operator $T\in B(H^2)$ with $r(T)<1$ to $S$ through the Rosenblum operator $\tau_{T,S}$ and introduced the \emph{shift representation operators}
\[K_x=\sum_{n=0}^{\infty}T^nx\otimes e_n,\qquad x\in H^2.\]
The map $x\mapsto K_x$ is a one-to-one map from $H^2$ onto $\ker\tau_{T,S}$. Exploiting the analytic and geometric structure encoded in this kernel, \cite{FHJZ1} obtained an asymmetric weakening of the growth conditions in the Gelfand-Hille and Wermer theorems, an equivalent of the Invariant Subspace Problem via the injectivity of certain Hankel operators, and new results on cyclic vectors.

The present paper continues this program in a different direction: we extend the shift representation technique beyond the Hardy space setting and apply the resulting machinery to several long-standing problems in operator theory. The construction above uses only that $S$ has a left inverse, namely $S^*$; we therefore replace $(H^2,S)$ by $(R(\bD),M_z)$, where $R(\bD)$ is the Sobolev disk algebra and $M_z$ is multiplication by the coordinate function (see Section \ref{S:prelim}). This replacement is advantageous in three respects. First, the functions in $R(\bD)$ are holomorphic on $\bD$ and extend continuously to $\Dbar$, so point evaluations at points of $\Dbar$ are available---a feature that $H^2$ does not possess. Second, $M_z$ is an essentially normal weighted shift, strictly cyclic and indeed strongly strictly cyclic (Propositions \ref{P:R(D)} and \ref{P:Mz is strongly cyclic}); moreover, a subspace of $R(\bD)$ is invariant under $M_z$ if and only if it is an ideal, and the functions in every proper such subspace have a common zero in $\Dbar$ (Proposition \ref{P:Shields1}). Third, $M_z$ admits a canonical left inverse $B$, defined, on the orthonormal basis $\{e_n\}_{n=0}^{\infty}$ of $R(\bD)$ given by Proposition \ref{P:R(D)} (so $e_n=\beta_nz^n$), by $Be_0=0$ and $Be_{n+1}=\frac{\beta_{n+1}}{\beta_n}e_n$; in particular, $B{\mathbf 1}=0$, with ${\mathbf 1}$ the constant function $1$.

These facts suggest the following two-parameter generalization of the shift representation operators. For $A,B,T\in B(\H)$ with $BA=I$ and $r(T)r(B)<1$, and $x,y\in\H$, the series
\[K_{x,y}:=\sum_{n=0}^{\infty}T^nx\otimes B^{*n}y\]
converges in norm, so that $K_{x,y}\in B(\H)$ (Lemma \ref{L:bounded}); for $B=S^*$ and $y=e_0$ it reduces to $K_x$.

Section \ref{S:prelim} collects the basic properties of the operators $K_{x,y}$ and the preliminaries on the Sobolev disk algebra $R(\bD)$. Our starting point is the criterion that $T$ is intransitive whenever $0\neq\ol{K_{x,y}(\N)}\subsetneq\H$ for some invariant linear manifold $\N$ of $A$ (Lemma \ref{L:k_{x,y}}); applied with $A=M_z$, $B$ the canonical left inverse of $M_z$ and $y={\mathbf 1}$, it yields the first of our main theorems, stated below as Theorem \ref{A}.

\begin{theoremi}\label{A}
Let $\H=R(\bD)$, and let $T\in B(\H)$ with $r(T)<1$. Suppose that there exist a nonzero vector $g\in\H$ and a nontrivial invariant linear manifold $\M$ of $M_z$ such that for every $f\in\M$, the zero set
\[Z(f(T)g):=\{z_0\in\Dbar:[f(T)g](z_0)=0\}\]
is nonempty. Then $T$ is intransitive.
\end{theoremi}

The proof of Theorem \ref{A} explicitly constructs a nontrivial invariant subspace of $T$, namely $\overline{K_{g,{\mathbf 1}}(\M)}$ (Remark \ref{R:main theorem}).

In Section \ref{S:ideal} we distill from Theorem \ref{A} the \emph{ideal property} of operators on $R(\bD)$ (Definition \ref{D:ideal property}); Theorem \ref{T:operator with i.p. is intransitive} shows that every operator with this property is intransitive, and the notion turns out to be universal.

\begin{theoremi}\label{B}
Every intransitive operator $T\in B(\H)$ is unitarily equivalent to an operator $\td T\in B(R(\bD))$ which has the ideal property.
\end{theoremi}

Combining Theorems \ref{T:operator with i.p. is intransitive} and \ref{B}, we see that $T\in B(\H)$ is intransitive if and only if it is unitarily equivalent to an operator on $R(\bD)$ with the ideal property; natural examples come from multiplication at a single point (Example \ref{E:one point}). Moreover, whenever $\sigma_p(T)=\varnothing$, the ideal property is inherited by the norm-closed algebra generated by $T$ and, under a natural commutant condition, by its strong operator closure (Propositions \ref{P:ideal property v.s. Banach algebra} and \ref{P:ideal property v.s. commutant}).

Section \ref{S:Pearcy} is devoted to a problem of Pearcy \cite{Pea78}: if $N$ is a diagonal normal operator with eigenvalues dense in $\bD$, show that every operator $N+F$ with $F$ of rank one has a nontrivial invariant subspace. Our methods apply to rank-one perturbations $D+f\otimes g$ of diagonal operators on $R(\bD)$: as soon as either $f$ or $g$ is not a cyclic vector of $D$, the operator $D+f\otimes g$ has the ideal property and is therefore intransitive (Lemma \ref{L:Pearcy} and Corollary \ref{C:Pearcy}).

Theorem \ref{T:FHJZ1} (a generalization of Theorem 4.4 of \cite{FHJZ1}) shows that $T\in B(\H)$ is intransitive if and only if there exist nonzero vectors $\xi,\eta\in\H$ for which the generating function
\[\varphi(z):=\sum_{n=0}^{\infty}\langle T^n\xi,\eta\rangle z^n\]
is rational. Thus $T$ is intransitive whenever the moments $\langle T^n\xi,\eta\rangle$ are of the form $p_1(n)\lambda_1^n+\cdots+p_m(n)\lambda_m^n$ with $p_1,\dots,p_m$ polynomials (Theorem \ref{T:polynomial moments}); moreover, for finite-rank perturbations of diagonal operators satisfying suitable orthogonality conditions, the rationality of $\varphi$ can be verified explicitly (Theorems \ref{T:H2 perturbation} and \ref{T:finite rank perturbations}).

Section \ref{S:Halmos} is devoted to Halmos's third problem \cite{Hal70}, which has remained open for more than fifty years: is the inverse of an invertible intransitive operator necessarily also intransitive? Some partial positive results are due to Douglas, Foias and Pearcy \cite{DFP05} and to Jiang, Fang and Ji \cite{JFJ25}. We say that $T\in B(\H)$ has a \emph{conjugate bi-geometric form} if, with respect to some nontrivial decomposition $\H=\mathfrak{M}\oplus\mathfrak{M}^{\perp}$, it is of the form $\left[\begin{smallmatrix}T_1&T_{12}\\0&T_2\end{smallmatrix}\right]$ with $T_1^*,T_2\in B_1(O_\delta)$ for some $\delta>0$, where $O_\delta:=\{z\in\mathbb C:|z|<\delta\}$ and $B_1(\Omega)$ denotes the Cowen--Douglas class, consisting of the operators $X\in B(\H)$ such that $\Omega\subset\sigma(X)$, $\ran(X-z)=\H$, $\bigvee_{z\in\Omega}\ker(X-z)=\H$ and $\dim\ker(X-z)=1$ for every $z\in\Omega$. We show that it suffices to solve Halmos's third problem for this class of operators.

\bigskip

\begin{theoremi}\label{C}
The following two statements are equivalent.
\begin{enumerate}
\item[(1)] For any invertible operator $T\in B(\H)$ with a nontrivial invariant subspace, $T^{-1}$ has a nontrivial invariant subspace.
\item[(2)] For any invertible operator $T\in B(\H)$ with a conjugate bi-geometric form, $T^{-1}$ has a nontrivial invariant subspace.
\end{enumerate}
\end{theoremi}

There is an instructive analogy with the invariant subspace problem for subnormal operators. In \cite{Sarason72}, Sarason observed that the problem for a general subnormal operator can be reduced to a special case, namely that of the operator of multiplication by $z$ on the closure of the polynomials in $L^2(\mu)$ for a suitable finite positive measure $\mu$; based on this reduction, Brown \cite{Brown78} proved that every subnormal operator has a nontrivial invariant subspace. Theorem \ref{D} plays a role analogous to that of Sarason's reduction in the study of Halmos's third problem: it shows that the general problem is equivalent to its restriction to the special class of invertible operators with conjugate bi-geometric form.

For invertible operators with conjugate bi-geometric form we obtain an explicit formula for $T^{-1}$ (Proposition \ref{P:form of inverse}), which yields new classes of operators whose inverses are intransitive: if the off-diagonal part is of rank one, say $T_{12}=e\otimes f$, then $T$ is invertible if and only if $e\notin\ran A$ and $f\notin\ran B^*$, and then $T^{-1}$ is intransitive (Corollary \ref{C:rank-one}, extending the bilateral shift). In particular, we obtain the following criterion.

\begin{theoremi}\label{D}
Let $T\in B(\H)$ be an invertible operator. Suppose that there is a nontrivial projection $P$ such that $PTP=TP$ and such that $PT(I-P)$ is of rank one. Then $T^{-1}$ is intransitive.
\end{theoremi}

The hypotheses of Theorem \ref{D} are easy to verify in concrete situations, so the theorem gives a readily applicable affirmative answer to Halmos's third problem.

Finally, in Section \ref{S:tridiagonal} we consider tridiagonal operators: although it remains open whether every tridiagonal operator has a nontrivial closed invariant subspace, we exhibit a class of tridiagonal operators that are intransitive and admit infinite decreasing chains of invariant subspaces (Theorem \ref{T:tridiagonal}).


\bigskip

\noindent \textbf{Structure of the paper.} Section \ref{S:prelim} introduces the operators $K_{x,y}$ and the preliminaries on $R(\bD)$; the proofs of Theorems \ref{A} and \ref{B} are given in Sections \ref{S:main} and \ref{S:ideal}, where the ideal property is also studied. Sections \ref{S:Pearcy}, \ref{S:Halmos} and \ref{S:tridiagonal} then treat, respectively, Pearcy's problem (including the rationality criterion for intransitivity), Halmos's third problem (Theorems \ref{C} and \ref{D}) and tridiagonal operators.

\section{The operators \texorpdfstring{$K_{x,y}$}{K\_\{x,y\}} and the Sobolev disk algebra}\label{S:prelim}

In this section we establish the basic properties of the operators $K_{x,y}$ introduced above and collect the preliminaries on the Sobolev disk algebra $R(\bD)$, which provides the setting for all our applications. Throughout the first part of this section, $A,B,T$ denote operators in $B(\mathcal H)$ with $BA=I$, so that $A$ is a right inverse of $B$ and $B$ is a left inverse of $A$.

\begin{lem}\label{L:bounded}
Suppose $B,T\in B(\mathcal H)$ satisfy $r(T)r(B)<1$. For $x,y\in \mathcal H$, define
$$
K_{x,y}:=\sum_{n=0}^\infty T^n x\otimes B^{*n} y.
$$
Then $K_{x,y}\in B(\mathcal H)$.
\end{lem}

\begin{proof}
For $z\in \mathcal H$ with $\|z\|=1$,
$$
\sum_{n=0}^\infty \|T^nx\|\,|\langle z,B^{*n}y\rangle|
\leq \sum_{n=0}^\infty \|T^n\|\|x\|\,\|B^{*n}\|\|y\|.
$$
Since $\|B^{*n}\|=\|B^n\|$ and $r(T)r(B)<1$, choose $\mu\in(r(T)r(B),1)$. By the spectral radius formula, there exists $N$ such that for all $n>N$,
$$
\|T^n\|\,\|B^{*n}\|\leq \mu^n.
$$
Thus $K_{x,y}$ is a well-defined linear operator on $\mathcal H$, and for general $z\in\mathcal H$,
$$
\|K_{x,y}(z)\|
\leq \left(\sum_{n=0}^N \|T^n\|\|x\|\|B^{*n}\|\|y\|
+\sum_{n=N+1}^\infty \mu^n\|x\|\|y\|\right)\|z\|.
$$
Hence $K_{x,y}$ is bounded.
\end{proof}

\begin{lem}\label{L:compact}
Under the hypotheses of Lemma \ref{L:bounded}, the operator $K_{x,y}$ is compact, that is, $K_{x,y}\in K(\mathcal H)$.
\end{lem}

\begin{proof}
Let $\{z_k\}$ be a sequence in the unit ball of $\mathcal H$ such that $z_k\to 0$ weakly. We show that $K_{x,y}(z_k)\to 0$ in norm.

Fix $\varepsilon>0$. Since $r(T)r(B)<1$, choose $\mu\in(r(T)r(B),1)$. Then there exists $N$ such that for all $n>N$,
$$
\|T^n\|\,\|B^{*n}\|\leq \mu^n,
$$
and moreover
$$
\sum_{n=N+1}^\infty \mu^n\|x\|\|y\|<\frac{\varepsilon}{2}.
$$
For each $k$,
$$
\|K_{x,y}(z_k)\|
\leq \sum_{n=0}^N \|T^nx\|\,|\langle z_k,B^{*n}y\rangle|
+\sum_{n=N+1}^\infty \|T^nx\|\,|\langle z_k,B^{*n}y\rangle|.
$$
The tail is bounded by
$$
\sum_{n=N+1}^\infty \|T^n\|\|x\|\|B^{*n}\|\|y\|
\leq \sum_{n=N+1}^\infty \mu^n\|x\|\|y\|
< \frac{\varepsilon}{2}.
$$
Since $z_k\to 0$ weakly, for each fixed $0\leq n\leq N$,
$$
\langle z_k,B^{*n}y\rangle\to 0.
$$
Thus the finite sum tends to $0$, so for sufficiently large $k$,
$$
\sum_{n=0}^N \|T^nx\|\,|\langle z_k,B^{*n}y\rangle|
< \frac{\varepsilon}{2}.
$$
Hence $\|K_{x,y}(z_k)\|<\varepsilon$ for large $k$, so $K_{x,y}\in K(\mathcal H)$.
\end{proof}

\begin{lem}\label{L:kernel of Rosenblum}
Let $A,B,T\in B(\mathcal H)$ satisfy $r(T)r(B)<1$ and $BA=I$. For every $y\in \ker A^*$ and every $x\in \mathcal H$, we have $K_{x,y}\in \ker \tau_{T,A}$.
\end{lem}

\begin{proof}
Since $BA=I$, taking adjoints gives $A^*B^*=I$. Also,
$$
T K_{x,y}
=\sum_{n=0}^\infty T^{n+1}x\otimes B^{*n}y.
$$
For each rank-one term,
$$
(T^n x\otimes B^{*n}y)A
= T^n x\otimes A^*B^{*n}y.
$$
Because $A^*y=0$, the $n=0$ term vanishes. For $n\geq 1$,
$$
A^*B^{*n}y
= (A^*B^*)B^{*(n-1)}y
= B^{*(n-1)}y.
$$
Therefore
$$
K_{x,y}A
=\sum_{n=1}^\infty T^n x\otimes B^{*(n-1)}y
=\sum_{m=0}^\infty T^{m+1}x\otimes B^{*m}y
=T K_{x,y}.
$$
Thus $K_{x,y}A=TK_{x,y}$, so $K_{x,y}\in \ker \tau_{T,A}$.
\end{proof}

\begin{lem}\label{L:k_{x,y}}
Let $A,B,T\in B(\mathcal H)$ satisfy $r(T)r(B)<1$ and $BA=I$. Suppose $0\neq y\in \ker A^*$, $0\neq x\in \mathcal H$, and $\mathcal N$ is a nontrivial invariant linear manifold of $A$ such that
$$
0\neq \overline{K_{x,y}(\mathcal N)}\subsetneq \mathcal H.
$$
Then $T$ is intransitive.
\end{lem}

\begin{proof}
By Lemma \ref{L:kernel of Rosenblum},
$$
T(K_{x,y}(\mathcal N))=K_{x,y}A(\mathcal N)\subseteq K_{x,y}(\mathcal N).
$$
Since $T$ is continuous,
$$
T\left(\overline{K_{x,y}(\mathcal N)}\right)
\subseteq \overline{K_{x,y}(\mathcal N)}.
$$
By assumption, this closed subspace is nonzero and proper. Hence $T$ has a nontrivial closed invariant subspace and is therefore intransitive.
\end{proof}

\bigskip

We now turn to the setting in which these operators will be applied: $\mathcal H=R(\bD)$, $A=M_z$, $B$ the canonical left inverse of $M_z$ (constructed at the end of this section), and $y=\mathbf 1$. The remainder of this section collects the necessary facts about $R(\bD)$ and $B$.

Let $\mathbb D$ be the unit disc in the complex plane, and let $W^{22}(\mathbb D)$ denote the Sobolev space
\[
W^{22}(\mathbb D):=\left\{f\in L^2(\mathbb D, dm): D^\alpha f\in L^2(\mathbb D, dm) \text{ for all } |\alpha|\leq 2\right\},
\]
where $dm$ denotes planar Lebesgue measure.

For $f,g\in W^{22}(\mathbb D)$, define
\[
\langle f,g\rangle:=\sum_{|\alpha|\leq 2}\int_{\mathbb D} D^\alpha f\,\overline{D^\alpha g}\,dm.
\]
Endowed with the inner product above, $W^{22}(\mathbb D)$ is a Hilbert space, and under an equivalent norm it is a unital Banach algebra. By the Sobolev embedding theorem, $f\in W^{22}(\mathbb D)$ implies $f\in C(\overline{\mathbb D})$ and $\|f\|_{C(\overline{\mathbb D})}\leq M\|f\|_{W^{22}(\mathbb D)}$ for some constant $M$.

For $f\in W^{22}(\mathbb D)$, define the multiplication operator $M_f$ on $W^{22}(\mathbb D)$ by
\[
M_f g:=fg,\qquad g\in W^{22}(\mathbb D).
\]
Let
\[
W(\mathbb D):=\{M_f : f\in W^{22}(\mathbb D)\}.
\]
Then $W(\mathbb D)$ is a strictly cyclic operator algebra with strictly cyclic vector $e(z)=1$. We write $\sigma_e(T)$ and $\mathrm{ind}(T)$ for the essential spectrum and the Fredholm index of an operator $T$, and $\sigma_{\mathrm{lre}}(T)$ for its left-right essential spectrum, the intersection of the left and right essential spectra. It is well known that
\[
\sigma(M_z)=\sigma_{\mathrm{lre}}(M_z)=\overline{\mathbb D},
\]
and
\[
\{M_z\}'=W(\mathbb D).
\]
Let $\mathcal A^a(M_z)$ be the closed algebra generated by rational functions of $M_z$ with poles outside $\overline{\mathbb D}$, and set $R(\mathbb D):=\mathcal A^a(M_z)e$. Since $W^{22}(\mathbb D)$ is strictly cyclic, there exist positive constants $N,K$ such that for every $f\in W^{22}(\mathbb D)$,
\begin{equation}\label{Eq:equivalent norms}
N\|f\|_{W^{22}(\mathbb D)}\leq \|M_f\|\leq K\|f\|_{W^{22}(\mathbb D)}.
\end{equation}
Thus $R(\mathbb D)$ is the closure in $W^{22}(\mathbb D)$ of all rational functions with poles outside $\overline{\mathbb D}$; we refer the reader to \cite{Lam71} for details.

The following proposition is \cite[Proposition 4.5.1]{JW2}.

\begin{prop}\label{P:multiplication operator}
Denote
\[
M_z(\mathbb D):=M_z|_{R(\mathbb D)}.
\]
Then:
\begin{enumerate}
\item $\sigma(M_z(\mathbb D))=\overline{\mathbb D}$ and $\sigma_e(M_z(\mathbb D))=\partial \mathbb D$; moreover, for every $z_0\in \mathbb D$,
\[
\ker(M_z(\mathbb D)-z_0)=\{0\}\quad\text{and}\quad \operatorname{ind}(M_z(\mathbb D)-z_0)=-1.
\]

\item The maximal ideal space of $\{M_z(\mathbb D)\}'$ is homeomorphic to $\overline{\mathbb D}$. The homomorphism $k_{z_0}$ corresponding to $z_0\in \overline{\mathbb D}$ is given by
\[
k_{z_0}^*(M_f(\mathbb D))=f(z_0)=\langle f,k_{z_0}\rangle
\]
for some $k_{z_0}\in R(\mathbb D)$, and $k_{z_0}\in \ker (M_z(\mathbb D)-z_0)^*$.

\item $\{M_z(\mathbb D)\}'=\mathcal A^a(M_z(\mathbb D))=\{M_f(\mathbb D):f\in R(\mathbb D)\}$, where $M_f(\mathbb{D}):=M_f|_{R(\mathbb{D})}$.
\end{enumerate}
\end{prop}

From now on, we call $R(\mathbb D)$ the \textbf{Sobolev disk algebra}. For simplicity, we shall write $M_f$ instead of $M_f(\mathbb D)$.

\begin{defn}\label{D:beta}
Let $\{\beta_n\}$ be a sequence of positive numbers. Consider the space of sequences $f=\{f_n\}$ such that
\[
\|f\|_\beta^2:=\sum_{n=0}^\infty |f_n|^2\beta_n^2<+\infty.
\]
We formally write
\[
f(z)=\sum_{n=0}^\infty f_n z^n.
\]
Set
\[
H^2(\beta):=\{f:\{f_n\}\in \ell^2(\beta_n)\}.
\]
Let $H^\infty(\beta)$ denote the set of formal series $\phi(z)=\sum_{n=0}^\infty \widehat\phi(n)z^n$ such that $\phi H^2(\beta)\subset H^2(\beta)$.
\end{defn}

For the convenience of the reader, let us recall some results from \cite{Shi}.

\begin{thm}\cite[p62, Theorem 3]{Shi}\label{commutant}
If $A$ is an operator on $H^2(\beta)$ that commutes with $M_z$, then $A=M_\phi$ for some $\phi\in H^\infty(\beta)$.
\end{thm}

\begin{prop}\cite[p94, Proposition 31; p94, Corollary 1; p95, Corollary 2]{Shi}\label{P:Shields1}
Let $T$ be an injective unilateral shift represented as $M_z$ on $H^2(\beta)$. Then $T$ is strictly cyclic if and only if $H^2(\beta)=H^\infty(\beta)$. In that case:
\begin{enumerate}
\item there exists $c>0$ such that $c\|f\|_\infty\leq \|f\|\leq \|f\|_\infty$ for all $f\in H^2(\beta)$;
\item the maximal ideal space of $H^\infty(\beta)$ is the closed disc $\Delta(T)^-=\{z\in\mathbb C: |z|\leq r(T)\}$;
\item a subspace is invariant under $M_z$ if and only if it is an ideal. Moreover, the functions in any proper invariant subspace have a common zero in $\Delta(T)^-$.
\end{enumerate}
\end{prop}

The following proposition collects the basic properties of $R(\mathbb D)$.

\begin{prop}\cite[Proposition 4.5.3]{JW2}\label{P:R(D)}
\begin{enumerate}
\item The Hilbert space $R(\mathbb D)$ has an orthonormal basis $\{e_n\}_{n=0}^\infty$, where $e_n=\beta_n z^n$ and
\[
\beta_n=\left[\frac{n+1}{(3n^4-n^2+2n+1)\pi}\right]^{1/2},\qquad n=0,1,2,\dots
\]

\item If $f(z)=\sum_{n=0}^\infty c_n z^n$ is analytic in $\mathbb D$, then $f\in R(\mathbb D)$ if and only if
\[
\sum_{n=0}^\infty \frac{|c_n|^2}{\beta_n^2}<+\infty.
\]
Moreover, if $f\in R(\mathbb D)$, then the partial sums $S_n=\sum_{k=0}^n c_k z^k$ converge uniformly to $f$ on $\overline{\mathbb D}$.

\item The operator $M_z$ is an essentially normal unilateral weighted shift, with
\[
M_z e_n=\frac{\beta_n}{\beta_{n+1}}e_{n+1},\qquad n=0,1,2,\dots
\]
\end{enumerate}
\end{prop}

\begin{rem} With the notation of Definition \ref{D:beta}, $R(\mathbb D)$ is actually $H^2(\beta^{-1})$, where
\[
\beta^{-1}:=\left\{\left[\frac{(3n^4-n^2+2n+1)\pi}{n+1}\right]^{1/2}\right\}_{n=0}^\infty.
\]
\end{rem}

\begin{prop}\label{P:Mz is strongly cyclic}
The operator $M_z$ is strongly strictly cyclic on $R(\mathbb D)$.
\end{prop}
\begin{proof}
Let $\{\beta_n\}_{n=0}^\infty$ be as in Proposition \ref{P:R(D)}. For each $n$, set
\[
\widetilde\beta_n:=(n+1)^{3/2}.
\]
It is easy to check that
\[
\frac{\beta_n^{-1}}{2\sqrt{\pi}}\leq \widetilde\beta_n\leq 2\beta_n^{-1},\qquad \forall n\ge 0.
\]
Hence $H^2(\widetilde\beta)=H^2(\beta^{-1})=R(\mathbb D)$. By \cite[p99, Example 1]{Shi}, $M_z$ on $H^2(\widetilde\beta)$ is strongly strictly cyclic. By \cite[p60, Theorem 2'(b)]{Shi}, the operators $M_z$ on $H^2(\widetilde\beta)$ and on $R(\mathbb D)$ are similar. Therefore $M_z$ is strongly strictly cyclic on $R(\mathbb D)$.
\end{proof}

We shall also need the canonical left inverse of $M_z$, which will play the role of $S^*$ in the Hardy space setting. Let $B\in B(R(\mathbb D))$ denote the operator defined by
\[
B e_0=0,\qquad B e_{n+1}=\frac{\beta_{n+1}}{\beta_n} e_n,\quad n=0,1,2,\dots
\]
Then $B M_z=I$. We call $B$ the canonical left inverse of $M_z$. Since $\{\beta_n\}$ is decreasing,
\[
\|B\|=\sup_{n\ge 0}\frac{\beta_{n+1}}{\beta_n}\le 1,
\]
so $r(B)\le 1$. In particular, $r(T)r(B)<1$ for every $T\in B(R(\mathbb D))$ with $r(T)<1$, so that the results of the first part of this section apply to the pair $(A,B)=(M_z,B)$ whenever $r(T)<1$.

%

\section{An intransitivity criterion}\label{S:main}

In this section we prove the first of our main theorems, Theorem \ref{A}, an intransitivity criterion on the Sobolev disk algebra. We begin with the following elementary lemma.

\begin{lem}\label{L:polynomial}
Let $\varphi(z)=\sum_{k=0}^{\infty}c_k z^k\in R(\mathbb D)$ satisfy $\varphi(\lambda_0)=0$ for some $\lambda_0\in \overline{\mathbb D}$. Then for every $\varepsilon>0$, there exists a polynomial $p\in \mathbb C[z]$ such that $\|p-\varphi\|<\varepsilon$ and $p(\lambda_0)=0$.
\end{lem}
\begin{proof}
Set $\varepsilon_1=\frac{\varepsilon}{4}>0$. By Proposition \ref{P:R(D)} and the convergence of $\sum_{n\ge0}\beta_n^2$, there exists $N\in\mathbb N$ such that
\[
\left\|\sum_{k=N+1}^{\infty}c_k z^k\right\|^2
=\sum_{k=N+1}^\infty \frac{|c_k|^2}{\beta_k^2}<\varepsilon_1^2
\qquad\text{and}\qquad
\left(\sum_{k=N+1}^\infty \beta_k^2\right)^{1/2}\le \frac{1}{\|\mathbf 1\|}.
\]
Then, since $|\lambda_0|\le 1$,
\[
\left|\sum_{k=N+1}^{\infty}c_k\lambda_0^k\right|
\leq \sum_{k=N+1}^{\infty}|c_k|\leq \left(\sum_{k=N+1}^\infty\frac{|c_k|^2}{\beta_k^2}\right)^{1/2}
\left(\sum_{k=N+1}^\infty \beta_k^2\right)^{1/2} <\frac{\varepsilon_1}{\|\mathbf 1\|}.
\]
Define
\[
p(z):=\sum_{k=0}^{N} c_k z^k-\sum_{k=0}^{N}c_k\lambda_0^k.
\]
Then clearly $p(\lambda_0)=0$. Finally, since $\varphi(\lambda_0)=0$, we have
\[
\sum_{k=0}^N c_k\lambda_0^k=-\sum_{k=N+1}^{\infty}c_k\lambda_0^k;
\]
hence
\[
\|\varphi-p\|
=\left\|\sum_{k=N+1}^{\infty}c_k z^k-\sum_{k=N+1}^{\infty}c_k\lambda_0^k\right\|
\le \varepsilon_1+\varepsilon_1
=\frac{\varepsilon}{2}<\varepsilon.
\]

\end{proof}

We now turn to the proof of Theorem \ref{A}.

\begin{proof}[\bf Proof of Theorem \ref{A}]

\noindent{\bf Step 1.} Let $y={\mathbf 1}$, $x=g$, $A=M_z$, and let $B$ be the canonical left inverse of $A$ (so $BA=I$ and $B{\mathbf 1}=0$). Note that $y\in \ker A^*\cap \ker B$.

We claim that for $f(z)=\sum_{k=0}^\infty a_k z^k\in R(\mathbb D)$,
\[
\|{\mathbf 1}\|^2f(T)g=\sum_{n=0}^\infty \langle f,(B^*)^n {\mathbf 1}\rangle\, T^n g.
\]
Indeed,
\begin{align*}
\langle f,(B^*)^n {\mathbf 1}\rangle
&=\sum_{k=0}^\infty a_k \langle z^k,(B^*)^n {\mathbf 1}\rangle \\
&=\sum_{k=0}^\infty a_k \langle A^k {\mathbf 1},(B^*)^n {\mathbf 1}\rangle \\
&=\sum_{k=0}^{n-1} a_k \langle A^k {\mathbf 1},(B^*)^n {\mathbf 1}\rangle
+a_n\|{\mathbf 1}\|^2
+\sum_{k=n+1}^\infty a_k \langle A^k {\mathbf 1},(B^*)^n {\mathbf 1}\rangle \\
&=\sum_{k=0}^{n} a_k \langle B^{n-k}{\mathbf 1},{\mathbf 1}\rangle
+\sum_{k=n+1}^\infty a_k \langle {\mathbf 1},(A^*)^{k-n}{\mathbf 1}\rangle \\
&=a_n\|{\mathbf 1}\|^2.
\end{align*}
Thus $K_{x,y}f=\|\mathbf 1\|^2f(T)g$ for every $f\in R(\mathbb D)$; since $\|\mathbf 1\|\neq0$, the two sets
\[
\{f(T)g:f\in \mathcal M\}\qquad\text{and}\qquad \{K_{x,y}f:f\in \mathcal M\}
\]
have the same closed linear span. By Lemma \ref{L:k_{x,y}}, $\overline{K_{x,y}\mathcal M}\in \operatorname{Lat}(T)$.

\bigskip

It remains to verify the two hypotheses of Lemma \ref{L:k_{x,y}}, namely that the closed subspace $\overline{K_{x,y}\mathcal M}$ is nonzero and proper.

\medskip

\noindent{\bf Step 2.} $K_{x,y}\mathcal M\neq \{0\}$. If there exists $0\neq f\in \mathcal M$ with $f(T)g\neq 0$, then clearly $K_{x,y}\mathcal M\neq \{0\}$. Otherwise, for every $0\neq f\in \mathcal M$, we have $f(T)g=0$. Choose a unitary $U:H^2\to R(\mathbb D)$ and set $T_1:=U^*TU$. Then $T_1\in B(H^2)$ with $r(T_1)<1$. Since $f$ is holomorphic on a neighborhood of $\sigma(T_1)$, we have
\[
f(T_1)(U^*g)=U^*f(T)U(U^*g)=0.
\]
Equivalently, writing $f=\sum_{n=0}^\infty a_nz^n$ and $y:=\sum_{n=0}^\infty a_ne_n\in H^2$, we have
\[
K_{U^*g}(y)=f(T_1)(U^*g)=0,
\]
where
\[
K_{U^*g}:=\sum_{n=0}^\infty T_1^n(U^*g)\otimes e_n,
\]
and $\{e_n\}_{n=0}^\infty$ is the canonical ONB of $H^2$. Since $f\neq 0$, we have $y\neq0$, so $K_{U^*g}$ is not injective, i.e., $0\in\sigma_p(K_{U^*g})$. By \cite[Proposition 1.7]{FHJZ1}, $T_1$ has eigenvalues; hence $T_1$, and therefore $T$, is intransitive.

\medskip

\noindent{\bf Step 3.} $\overline{K_{x,y}\mathcal M}\neq \mathcal H$. It suffices to show that $\mathbf 1\notin \overline{K_{x,y}\mathcal M}=\overline{\{f(T)g:f\in \mathcal M\}}$. Suppose otherwise. Then there exist $f_n\in \mathcal M$ and $\lambda_n\in \overline{\mathbb D}$ such that $\varphi_n(\lambda_n)=0$ for $\varphi_n:=f_n(T)g$, and
\[
\lim_{n\to\infty}\|\varphi_n-\mathbf 1\|=0.
\]
By Lemma \ref{L:polynomial}, there exist polynomials $p_n$ with
\[
\|p_n-\varphi_n\|<\frac{1}{n},\qquad p_n(\lambda_n)=0.
\]
Hence $p_n=(z-\lambda_n)q_n$ for some polynomial $q_n$. Consequently,
\[
\lim_{n\to\infty}\|p_n-\mathbf 1\|=0.
\]
By Propositions \ref{P:Shields1} and \ref{P:Mz is strongly cyclic}, the norm of $R(\mathbb D)$ is equivalent to the supremum norm and $\|M_f\|\le C\|f\|$ for some constant $C$ and all $f\in R(\mathbb D)$; hence
\[
\lim_{n\to\infty} M_{p_n}=M_{\mathbf 1}.
\]
By Proposition \ref{P:multiplication operator}, $\ker(M_z^*-\overline{\lambda_n})\neq \{0\}$. Choose a unit vector $h_n\in \ker(M_z^*-\overline{\lambda_n})$. Then
\[
(z-\lambda_n)q_n h_n=(M_z-\lambda_n)q_n h_n\in \operatorname{ran}(M_z-\lambda_n)\subset \ker(M_z^*-\overline{\lambda_n})^\perp.
\]
So
\[
\langle (z-\lambda_n)q_n h_n,h_n\rangle=0.
\]
But
\[
\left|\langle (z-\lambda_n)q_n h_n,h_n\rangle-1\right|
=\left|\langle (M_{p_n}-M_{\mathbf 1})h_n,h_n\rangle\right|
\le \|M_{p_n}-M_{\mathbf 1}\|\to 0,
\]
since $\|h_n\|=1$ and $M_{\mathbf 1}=I$, a contradiction. Hence $\mathbf 1\notin \overline{K_{x,y}\mathcal M}$.
Therefore $T$ is intransitive.
\end{proof}

\begin{rem}\label{R:main theorem}
The proof of Theorem \ref{A} is constructive. Indeed, take $x=g$, $y={\mathbf 1}$, and $A=M_z$ with $B$ its canonical left inverse; then, as Lemma \ref{L:k_{x,y}} shows, the closed subspace $\overline{K_{x,y}(\mathcal M)}$ is invariant under $T$, and it has the same closed linear span as the set $\{f(T)g:f\in\mathcal M\}$. The two verifications above show that this subspace is neither $\{0\}$ nor $\mathcal H$. Thus the nontrivial invariant subspace of $T$ produced by the proof is determined by the given $M_z$-invariant manifold $\mathcal M$ and the vector $g$.
\end{rem}

\section{Operators with the ideal property}\label{S:ideal}

Motivated by Theorem \ref{A}, we isolate the following property of operators on $R(\bD)$.

\begin{defn}\label{D:ideal property}
Let $\mathfrak I$ be the set of all nontrivial ideals of $R(\mathbb D)$. We say that $T\in B(R(\mathbb D))$ has the \emph{ideal property} if there exist $0<\mu<1/r(T)$ (here we adopt the convention that $1/0=\infty$), a nontrivial ideal $\mathcal M\in\mathfrak I$, and a nonzero vector $g\in R(\mathbb D)$ such that for every $f\in\mathcal M$, the closed ideal $I(f(\mu T)g)$ generated by $f(\mu T)g$ is proper.
\end{defn}

\begin{lem}\label{L:invertible}
For $f\in R(\mathbb D)$, we have $f(z)\neq 0$ for all $z\in\overline{\mathbb D}$ if and only if $f$ is invertible in $R(\mathbb D)$.
\end{lem}
\begin{proof}
Suppose that $f$ is invertible in $R(\mathbb{D})$; then clearly $f(z)\neq 0$ for all $z\in \overline{\mathbb{D}}$.
Conversely, suppose that $f(z)\neq 0$ for all $z\in\overline{\mathbb D}$.
Let $\mathcal M:=\overline{fR(\mathbb D)}$. Then $\mathcal M$ is a closed invariant subspace of $M_z$. If $\mathcal M\neq R(\mathbb D)$, then by Proposition \ref{P:Shields1}(3), all functions in $\mathcal M$ have a common zero; in particular, $f=f\cdot {\mathbf 1}$ has a zero in $\overline{\mathbb D}$, contradiction. Hence $\mathcal M=R(\mathbb D)$, so there exists $g\in R(\mathbb D)$ with $\|fg-\mathbf{1}\|<1/2$, implying $f$ is invertible.
\end{proof}

\begin{thm}\label{T:operator with i.p. is intransitive}
Let $T\in B(R(\mathbb D))$. If $T$ has the ideal property, then $T$ is intransitive.
\end{thm}
\begin{proof}
Since $T$ has the ideal property, there exist $0<\mu<1/r(T)$, a nontrivial ideal $\mathcal M$, and $g\in R(\mathbb D)$ such that for every $f\in\mathcal M$, the ideal $I(f(\mu T)g)$ is proper. By Proposition \ref{P:Shields1}(3), for every $f\in\mathcal M$, $Z(f(\mu T)g)\neq\varnothing$. Since $\mathcal M$ is invariant under $M_z$ and $r(\mu T)<1$, Theorem \ref{A} shows that $\mu T$ is intransitive. As $\mu T$ and $T$ have the same invariant subspaces, $T$ is intransitive.
\end{proof}

The following example shows that an operator which acts as multiplication at a single point has the ideal property.

\begin{eg}\label{E:one point}
Let $T\in B(R(\mathbb D))$ and $0\neq g\in R(\mathbb D)$. Suppose there exist $\alpha_0\in\mathbb D$ and $\omega_0\in\mathbb C$ such that
\[
(T^n g)(\alpha_0)=\omega_0^n,\qquad n\ge 0.
\]
Then $T$ has the ideal property, and hence is intransitive.

Indeed, let
\[
T_1:=\frac{T}{\|T\|+|\omega_0|+1},\qquad
\mathcal M:=\{f\in R(\mathbb D): f(\frac{\omega_0}{\|T\|+|\omega_0|+1})=0\}.
\]
Clearly $\mathcal M$ is an ideal. For $f=\sum_{n=0}^\infty a_n z^n\in\mathcal M$,
\[
[f(T_1)g](\alpha_0)
=\sum_{n=0}^\infty a_n \frac{(T^n g)(\alpha_0)}{(\|T\|+|\omega_0|+1)^n}
=\sum_{n=0}^\infty a_n\left(\frac{\omega_0}{\|T\|+|\omega_0|+1}\right)^n=0.
\]
Thus $T$ has the ideal property, and Theorem \ref{T:operator with i.p. is intransitive} applies.
\end{eg}

The operators with the ideal property are universal; this is Theorem \ref{B}, whose proof we now give.

\begin{proof}[\bf Proof of Theorem \ref{B}]
Let $\mathcal N$ be a nontrivial invariant subspace of $T$. With respect to $\mathcal H=\mathcal N\oplus \mathcal N^\perp$, write
\[
T=\begin{bmatrix}
T_{11}&T_{12}\\
0&T_{22}
\end{bmatrix}.
\]
Choose a subspace $\widetilde{\mathcal N}\subset R(\mathbb D)$ as follows:
\begin{itemize}
\item if $\dim\mathcal N=d<\infty$, let $\widetilde{\mathcal N}=\bigvee\{z,\dots,z^d\}$;
\item if $\dim\mathcal N^\perp=d<\infty$, let $\widetilde{\mathcal N}=(\bigvee\{{\mathbf 1},\dots,z^{d-1}\})^\perp$;
\item if both are infinite-dimensional, let $\widetilde{\mathcal N}=\bigvee\{z^{2n-1}:n\in\mathbb N\}$.
\end{itemize}
Then $\dim \mathcal N=\dim \widetilde{\mathcal N}$ and $\dim \mathcal N^\perp=\dim \widetilde{\mathcal N}^\perp$, so there exists a block diagonal unitary $U=U_1\oplus U_2$ mapping $\widetilde{\mathcal N}$ onto $\mathcal N$ and $\widetilde{\mathcal N}^\perp$ onto $\mathcal N^\perp$.

\bigskip

Set $\widetilde T:=U^*TU$. Then $\widetilde T\in B(R(\mathbb D))$ and $\widetilde{\mathcal N}$ is invariant under $\widetilde T$, hence under $p(\widetilde T)$ for every polynomial $p$. Put $\nu:=1/(\|\widetilde T\|+1)$, so that $r(\nu\widetilde T)<1$. Since $f(\nu\widetilde T)$ is a norm limit of polynomials in $\nu\widetilde T$ for every $f\in R(\mathbb D)$, the subspace $\widetilde{\mathcal N}$ is invariant under $f(\nu\widetilde T)$ as well.

Let
\[
\mathcal M:=\bigvee\{h\psi:\ h\in R(\mathbb D),\ \psi\in\widetilde{\mathcal N}\}
\]
be the closed ideal generated by $\widetilde{\mathcal N}$. Since $\widetilde{\mathcal N}\neq\{0\}$ and every element of $\widetilde{\mathcal N}$ vanishes at $0$, the ideal $\mathcal M$ is nontrivial. Fix a nonzero $g\in\widetilde{\mathcal N}$, and let $f\in\mathcal M$, say $f=\underset{k\to \infty}{\lim} f_k$ with $f_k\in \mathcal{L}\{h\psi:h\in R(\mathbb D),\psi\in\widetilde{\mathcal N}\}$, the linear manifold generated by $\{h\psi:h\in R(\mathbb D),\psi\in\widetilde{\mathcal N}\}$. By the multiplicativity of the analytic functional calculus, $f_k(\nu\widetilde T)g\in\widetilde{\mathcal N}$ for every $k$, and hence $f(\nu\widetilde T)g\in\widetilde{\mathcal N}$ by closedness; in particular, $[f(\nu\widetilde T)g](0)=0$. Thus $I(f(\nu\widetilde T)g)$ is a proper ideal of $R(\mathbb D)$, and $\widetilde T$ has the ideal property.
\end{proof}

\begin{prop}\label{P:ideal property v.s. Banach algebra}
Let $T\in B(R(\mathbb D))$ have the ideal property and $\sigma_p(T)=\varnothing$. Then every operator in $\mathcal A(T)$, the norm-closed algebra generated by $T$, also has the ideal property.
\end{prop}
\begin{proof}

\noindent{\bf Step 1.} Since $T$ has the ideal property, there exist $0<\mu<1/r(T)$, a nontrivial ideal $\mathcal M$, and a nonzero $g\in R(\mathbb D)$ such that the ideal $I(f(\mu T)g)$ is proper, equivalently $f(\mu T)g$ is not invertible, for every $f\in\mathcal M$.

\medskip

\noindent{\bf Step 2.} We claim that $f(\mu T)g\neq 0$ for all $0\neq f\in\mathcal M$. Indeed, if $f(\mu T)g=0$, then, since $\sigma(\mu T)$ is a compact subset of $\mathbb D$ and $f$ has only finitely many zeros there (counting multiplicity), we can write $f(z)=q(z)p(z)$ with $p$ a polynomial and $q$ holomorphic with no zeros on $\sigma(\mu T)$. By the spectral mapping theorem, $q(\mu T)$ is invertible, so $p(\mu T)g=0$. Here $p$ is nonconstant, since otherwise $g=0$; writing $p(z)=\prod_i(z-w_i)$ as a product of linear factors, we see that some $\mu T-w_i$ has a nonzero kernel, so that $w_i/\mu\in\sigma_p(T)$, a contradiction to $\sigma_p(T)=\varnothing$.

\medskip

\noindent{\bf Step 3.} Fix $h\in\mathcal M$ with $\|h\|=1$, and set $g_1:=h(\mu T)g$; by Step 2, $g_1\neq0$. Now take $A\in\mathcal A(T)$. There exist polynomials $p_n$ such that
\[
p_n(\mu T)\to \frac{A}{\|A\|+1}.
\]
For $f\in\mathcal M$, the analytic functional calculus gives
\[
f(p_n(\mu T))\to f\left(\frac{A}{\|A\|+1}\right).
\]
But for every polynomial $q$, the function $(q\circ p_n)h$ belongs to $\mathcal M$, so
\[
(q\circ p_n)(\mu T)h(\mu T)g=q(p_n(\mu T))g_1
\]
is not invertible. Approximating $f$ by polynomials in $R(\mathbb D)$ and using the closedness of the set of non-invertible elements, we conclude that $f(p_n(\mu T))g_1$ is not invertible for every $n$ with $\|p_n(\mu T)\|<1$. Since $p_n(\mu T)\to A/(\|A\|+1)$, the continuity of the analytic functional calculus gives
\[
f(p_n(\mu T))g_1\to f\left(\frac{A}{\|A\|+1}\right)g_1,
\]
which is therefore not invertible as well. Since $1/(\|A\|+1)<1/r(A)$, this shows that $A$ has the ideal property.

\end{proof}

\begin{prop}\label{P:ideal property v.s. commutant}
Let $T\in B(R(\mathbb D))$ have the ideal property and $\sigma_p(T)=\varnothing$, and suppose $\{T\}'=\mathcal A(T)^{\mathrm{SOT}}$, where $\mathcal A(T)^{\mathrm{SOT}}$ denotes the strong operator closure of $\mathcal A(T)$. Then every operator in $\mathcal A(T)^{\mathrm{SOT}}$ has the ideal property.
\end{prop}
\begin{proof}

 As in the proof of Proposition \ref{P:ideal property v.s. Banach algebra}, there exist $0<\mu<1/r(T)$, a nontrivial ideal $\mathcal M$, a nonzero $g\in R(\mathbb D)$ and a vector $g_1=h(\mu T)g\neq0$, with $h\in\mathcal M$ and $\|h\|=1$, such that $f(\mu T)g$ is not invertible for every $f\in\mathcal M$.

\medskip

 Take $A\in \{T\}'$ and $f\in\mathcal M$. By the analytic functional calculus, $f\left(\frac{A}{\|A\|+1}\right)$ commutes with $T$. Since $\{T\}'=\mathcal A(T)^{\mathrm{SOT}}$, there exist polynomials $p_n$ such that
\[
p_n(\mu T)\xrightarrow{\mathrm{SOT}} f\left(\frac{A}{\|A\|+1}\right).
\]
Then
\[
p_n(\mu T)h(\mu T)g=(p_n h)(\mu T)g,
\]
and since $p_n h\in\mathcal M$, the right-hand side is not invertible. Passing to the limit, which is a norm limit on the fixed vector $g_1$, and using the closedness of the set of non-invertible elements of $R(\mathbb D)$, we conclude that the operator $f\left(\frac{A}{\|A\|+1}\right)g_1$ is not invertible. Since $1/(\|A\|+1)<1/r(A)$, this shows that $A$ has the ideal property.

\end{proof}

\section{An application to Pearcy's problem}\label{S:Pearcy}

In this section we shall apply the methods developed above to a problem of Pearcy.

\begin{problem}[Pearcy]\cite{Pea78}\label{Pearcy Problem}
Suppose $N$ is a diagonal normal operator whose eigenvalues constitute a dense subset of the unit disc $\mathbb{D}$.
Show that every operator of the form $N+F$ has a nontrivial invariant subspace, where $F$ is an operator of rank one.
\end{problem}

Pearcy's problem has attracted much attention. For a rank-one perturbation $T=D+f\otimes g$ of a diagonal operator $D$, Foias, Jung, Ko and Pearcy \cite{FJKP07} obtained a nontrivial hyperinvariant subspace whenever $T$ is not a scalar multiple of the identity and the Fourier coefficients $\{\alpha_{n}\}$ and $\{\beta_{n}\}$ of $f$ and $g$ satisfy $\sum_{n}(|\alpha_{n}|^{2/3}+|\beta_{n}|^{2/3})<\infty$. Fang and Xia \cite{FX12} weakened this summability condition to $\sum_{n}(|\alpha_{n}|+|\beta_{n}|)<\infty$ and extended the result to finite-rank perturbations. Gallardo-Guti\'errez and Gonz\'alez-Do\~na \cite{GGD24} showed more recently that a vanishing Fourier coefficient of $f$ or $g$, or a logarithmic summability condition on the coefficients, already suffices, with analogous results for finite-rank perturbations. In contrast to these approaches, the method of this section exploits the ideal property of the operator and the rationality of the generating function of its moments.

We first need the following observation concerning the cyclic vectors of diagonal operators; recall that $x\in\H$ is called a \emph{cyclic vector} of $T\in B(\H)$ if $\bigvee\{T^nx:n\ge0\}=\H$.

\begin{prop}\label{P:cyclic vectors}
Suppose that $D\in B(\mathcal{H})$ is a diagonal operator and $x\in \mathcal{H}$.
Then $x$ is a cyclic vector of $D$ if and only if
$x$ is a cyclic vector of $D^*$.
\end{prop}

\begin{proof}
Suppose that $x=\sum_{n=0}^\infty \alpha_n e_n$ is not a cyclic vector of $D=\sum_{n=0}^\infty \lambda_n e_n\otimes e_n$. We show that $x$ is not a cyclic vector of $D^*$ either. If $\alpha_{n_1}=0$ for some $n_1\geq 0$, then $e_{n_1}\perp (D^*)^kx$ for all $k\ge0$, and we are done. Hence we may assume that $\alpha_n\neq 0$ for all $n\geq 0$.

Then there exists a unit vector $y=\sum_{n=0}^\infty \beta_n e_n\in \mathcal{H}$ such that $\langle D^kx, y\rangle=0$ for all $k\ge0$, that is,
\[
\sum_{n=0}^\infty \lambda_n^k \alpha_n \overline{\beta_n}=0,\qquad k\ge 0.
\]
Define $z=\sum_{n=0}^\infty \gamma_n e_n$, where $\gamma_n=\frac{\alpha_n^2}{|\alpha_n|^2}\overline{\beta_n}$. Then $z$ is a unit vector and for $k\geq 0$,
\[\langle (D^*)^kx,z\rangle=\sum_{n=0}^\infty \overline{\lambda_n}^k \alpha_n \overline{\gamma_n}=\overline{\sum_{n=0}^\infty \lambda_n^k \overline{\alpha_n}\gamma_n}=\overline{\sum_{n=0}^\infty \lambda_n^k \alpha_n \overline{\beta_n}}=0.\]
Hence $x$ is not a cyclic vector of $D^*$, and the same argument with $D^*$ in place of $D$ gives the converse.
\end{proof}

We now apply these observations to rank-one perturbations of diagonal operators on $R(\mathbb D)$.

\begin{lem}\label{L:Pearcy}
Let $T=D+f\otimes g\in B(R(\mathbb D))$, where $f,g\in R(\mathbb D)$ and
\[
D=\sum_{n=0}^\infty \alpha_n (\beta_nz^n)\otimes (\beta_nz^n)
\]
is the diagonal operator with respect to the orthonormal basis $\{\beta_nz^n\}_{n=0}^\infty$ of $R(\mathbb{D})$. Suppose that either $f$ or $g$ is not a cyclic vector of $D$. Then $T$ or $T^*$ has the ideal property. In particular, $T$ is intransitive.
\end{lem}

\begin{proof}
Suppose first that $g$ is not a cyclic vector of $D$. By Proposition \ref{P:cyclic vectors}, $g$ is not a cyclic vector of $D^*$ either, so there exists a unit vector $h\in R(\mathbb{D})$ such that
\[
\langle D^nh,g\rangle=\langle h,(D^*)^ng\rangle=0,\qquad n\in \mathbb{Z}_{\geq 0}.
\]
Put $\nu:=1/(\|T\|+\|D\|+1)$. We claim that
\[
T^nh=D^nh,\qquad n\ge0.
\]
Indeed, assuming $T^nh=D^nh$, we get
\[
T^{n+1}h=T(T^nh)=T(D^n h)=D^{n+1}h+\langle D^n h,g\rangle f=D^{n+1}h.
\]
Hence
\[
(\nu T)^nh=(\nu D)^nh,\qquad n\ge0.
\]
Set
\[
\mathcal M=(z-\nu\alpha_0)R(\mathbb{D}).
\]
Then $\mathcal M$ is an ideal of $R(\mathbb{D})$. Now for
\[
F=(z-\nu\alpha_0)F_1\in \mathcal M,
\]
we have
\[
F(\nu T)h=F(\nu D)h=\nu\,F_1(\nu D)(D-\alpha_0 I)h.
\]
Note that
\[
D-\alpha_0 I=\sum_{n=1}^\infty \alpha_n (\beta_nz^n)\otimes (\beta_nz^n),
\]
and hence $[(D-\alpha_0 I)h](0)=0$. Consequently,
\[
[F(\nu T)h](0)=0.
\]
Since $\nu<1/r(T)$ and $\mathcal M$ is a nontrivial ideal, this means that $T$ has the ideal property.

If instead $f$ is not a cyclic vector of $D$, then by Proposition \ref{P:cyclic vectors} again, $f$ is not a cyclic vector of $D^*$; since $T^*=D^*+g\otimes f$, the same argument shows that $T^*$ has the ideal property.

Finally, by Theorem \ref{T:operator with i.p. is intransitive} applied to $T$ or to $T^*$, the operator $T$ is intransitive.

\end{proof}

The following corollary records two simple instances in which the hypothesis of Lemma \ref{L:Pearcy} is satisfied.

\begin{cor}\label{C:Pearcy}
Let $T=D+f\otimes g$, where $D=\sum_{n=0}^\infty \alpha_n (\beta_nz^n)\otimes (\beta_nz^n)\in B(R(\mathbb{D}))$ is the diagonal operator with respect to the orthonormal basis $\{\beta_nz^n\}_{n=0}^\infty$ of $R(\mathbb{D})$, and $g=\sum_{n=0}^\infty c_nz^n\in R(\mathbb{D})$.
If $D$ has an eigenvalue of multiplicity greater than one, or if some coefficient $c_n$ vanishes, then $T$ has the ideal property. In particular, $T$ is intransitive.
\end{cor}

\begin{proof}
In either case $g$ is not a cyclic vector of $D$. Indeed, if $c_n=0$ for some $n$, then the basis vector $\beta_nz^n$ is orthogonal to $D^kg$ for all $k\ge0$. If $D$ has an eigenvalue of multiplicity greater than one, say $\alpha_i=\alpha_j$ with $i\neq j$, and if both $c_i$ and $c_j$ are nonzero, then the projections of the vectors $D^kg$, $k\ge0$, onto $\bigvee\{\beta_iz^i,\beta_jz^j\}$ all lie in the one-dimensional space spanned by $\frac{c_i}{\beta_i}\beta_iz^i+\frac{c_j}{\beta_j}\beta_jz^j$, so that $g$ is not cyclic either.

Consequently, by the proof of Lemma \ref{L:Pearcy}, $T$ has the ideal property. In particular, $T$ is intransitive.
\end{proof}

The results above reduce Pearcy's problem, in the setting of diagonal operators on $R(\bD)$, to the verification of the ideal property. To verify this property in more quantitative terms, we record a general criterion for intransitivity in terms of the generating function of the moments; it is a general form of \cite[Theorem 4.4]{FHJZ1}, in which the restrictions $\mathcal H=H^2$ and $r(T)<1$ are removed.

\begin{thm}\label{T:FHJZ1}
Let $T\in B(\mathcal H)$. Then $T$ is intransitive if and only if there exist nonzero $\xi,\eta\in\mathcal H$ such that
\[
\varphi(z):=\sum_{n=0}^\infty \langle T^n\xi,\eta\rangle z^n
\]
is a rational function.
\end{thm}
\begin{proof}
Since $\mathcal H$ is unitarily equivalent to $H^2$, and since unitary equivalence preserves intransitivity as well as the values $\langle T^n\xi,\eta\rangle$, we may assume that $\mathcal H=H^2$.

We first establish the theorem under the additional hypothesis $r(T)<1$. In this case, \cite[Theorem C]{FHJZ1} states that $T$ is intransitive if and only if there exist nonzero $\xi,\eta\in H^2$ such that the Hankel operator with symbol $\psi=\sum_{n=0}^\infty\langle T^n\xi,\eta\rangle z^{-n}$ is not injective; by \cite[Lemma 4.2]{FHJZ1}, this non-injectivity is in turn equivalent to the rationality of
\[
\sum_{n=0}^\infty \langle T^n\xi,\eta\rangle z^{n}.
\]
This proves the theorem when $r(T)<1$.

\smallskip

It remains to remove the hypothesis on the spectral radius. Choose $\alpha>0$ with $r(\alpha T)<1$. The operators $T$ and $\alpha T$ have the same invariant subspaces, so $T$ is intransitive if and only if $\alpha T$ is; moreover,
\[
\sum_{n=0}^\infty \langle (\alpha T)^n\xi,\eta\rangle z^n=\sum_{n=0}^\infty \langle T^n\xi,\eta\rangle (\alpha z)^n=\varphi(\alpha z).
\]
Since $\varphi(z)$ is rational if and only if $\varphi(\alpha z)$ is rational, the same nonzero vectors $\xi,\eta$ witness the rationality of the generating function for $T$ and for $\alpha T$. Applying the case already proved to $\alpha T$ completes the proof.
\end{proof}

\begin{thm}\label{T:polynomial moments}
Let $T\in B(\mathcal H)$ and let $\xi,\eta\in\mathcal H$ be nonzero. If
\[
\langle T^n\xi,\eta\rangle=p_1(n)\alpha_1^n+\cdots+p_m(n)\alpha_m^n,\qquad n\ge N_0,
\]
where $p_1,\dots,p_m$ are polynomials, then $T$ is intransitive.
\end{thm}
\begin{proof}
By Theorem~\ref{T:FHJZ1}, we only need to show that
\[
\varphi(z)=\sum_{n\geq N_0} (p_1(n)\alpha_1^n+\cdots+p_m(n)\alpha_m^n)z^n
\]
is a rational function. Therefore, it suffices to prove that $\sum_{n\geq N_0} p(n)\alpha^n z^n$ is a rational function, where $p$ is a polynomial. Furthermore, we may assume that $p(n)=n^k$ for some $k\in \mathbb{N}\cup\{0\}$. Note that
\[
\psi(z)=\sum_{n\geq N_0} \alpha^nz^n=\frac{(\alpha z)^{N_0}}{1-\alpha z}
\]
is a rational function. On the other hand, $n^k$ is a linear combination of the falling factorials $n(n-1)\cdots(n-j+1)$, $1\le j\le k$, together with a constant, and
\[
\sum_{n\geq N_0}n(n-1)\cdots(n-j+1)\,\alpha^n z^n=z^j\psi^{(j)}(z),\qquad j\ge0.
\]
Since $\psi$ is rational, so is each $z^j\psi^{(j)}(z)$. Thus $\sum_{n\geq N_0}n^k \alpha^n z^n$ is rational.
\end{proof}

We next consider rank-one perturbations of diagonal operators on the Hardy space, where the generating function can be analyzed more explicitly.

\begin{thm}\label{T:H2 perturbation}
Let $T=D+f\otimes g$, where $D=\sum_{n=0}^\infty \lambda_n z^n\otimes z^n\in B(H^2)$ and $f,g\in H^2$. Let $\bar D:=D-\lambda_0 I$ and $\bar T:=\bar D+f\otimes g$. If $\langle\bar T \mathbf{1},\mathbf{1}\rangle\neq0$, then the following are equivalent:
\begin{enumerate}
\item[(a)] $\displaystyle \sum_{n=0}^\infty \langle \bar T^n \mathbf{1},\mathbf{1}\rangle z^n$ is rational;
\item[(b)] $\displaystyle \sum_{n=0}^\infty \langle \bar D^n f,g\rangle z^n$ is rational.
\end{enumerate}
Moreover, if either (hence both) holds, then $T$ has a nontrivial invariant subspace. If $\langle\bar T\mathbf{1},\mathbf{1}\rangle=0$, then $T$ has a nontrivial invariant subspace as well.
\end{thm}
\begin{proof}
Let $x_n=\langle \bar{T}^n \mathbf{1},\mathbf{1}\rangle$ for $n\geq 1$. We claim that for $n\geq 2$,
\begin{equation}\label{E:xn}
x_n=\langle f,g\rangle x_{n-1}+\langle \bar{D}f, g\rangle x_{n-2}+\langle \bar{D}^2f, g\rangle x_{n-3}+\cdots+\langle \bar{D}^{n-2}f, g\rangle x_1.
\end{equation}
Indeed,
\[
x_1=\langle \bar{T} \mathbf{1},\mathbf{1}\rangle=\langle f,\mathbf{1}\rangle \langle \mathbf{1},g\rangle,
\]
\[
x_2=\langle \bar{T}^2 \mathbf{1},\mathbf{1}\rangle=\langle f,g\rangle \langle f,\mathbf{1}\rangle \langle \mathbf{1},g\rangle,
\]
\[
x_3=\langle \bar{T}^3 \mathbf{1},\mathbf{1}\rangle=(\langle \bar{D}f,g\rangle+(\langle f,g\rangle)^2) \langle f,\mathbf{1}\rangle \langle \mathbf{1},g\rangle,
\]
\[
x_4=\langle \bar{T}^4 \mathbf{1},\mathbf{1}\rangle=(\langle \bar{D}^2f,g\rangle+2\langle f,g\rangle \langle\bar{D}f,g\rangle+ (\langle f,g\rangle)^3) \langle f,\mathbf{1}\rangle \langle \mathbf{1},g\rangle,
\]
\begin{align*}
x_5=\langle \bar{T}^5 \mathbf{1},\mathbf{1}\rangle=\Bigl(&\langle \bar{D}^3f,g\rangle+2\langle f,g\rangle\langle \bar{D}^2f,g\rangle+3(\langle f,g\rangle)^2\langle\bar{D}f,g\rangle\\
&+\langle\bar{D}f,g\rangle^2+(\langle f,g\rangle)^4\Bigr) \langle f,\mathbf{1}\rangle \langle \mathbf{1},g\rangle.
\end{align*}
The cases $n=2,3,4,5$ are shown above; the general case follows by induction. Put $y_n=\langle \bar{D}^{n-1}f,g\rangle$ for $n\geq 1$. Then
\[
x_n=y_1x_{n-1}+y_2x_{n-2}+\cdots+y_{n-1}x_1,\quad n\geq 2.
\]
Thus
\[
\sum_{n=2}^\infty \langle \bar{T}^n \mathbf{1},\mathbf{1}\rangle z^n=\sum_{n=2}^\infty x_nz^n=\sum_{n=2}^\infty \left( \sum_{j=1}^{n-1} y_jx_{n-j}\right)z^n
\]
\[
=\sum_{n=2}^\infty \left( \sum_{i+j=n}y_jx_{i}\right)z^{i+j}=\sum_{i=1}^\infty x_iz^i \sum_{j=1}^\infty y_jz^j.
\]
Define $F(z)=\sum_{n=1}^\infty x_nz^n$ and $G(z)=\sum_{n=1}^\infty y_nz^n$. Then
\[
F(z)-x_1z=F(z)G(z).
\]
\smallskip

If $x_1\neq 0$, then
\[
F(z)=\frac{x_1z}{1-G(z)},\qquad G(z)=\frac{F(z)-x_1z}{F(z)}.
\]
It follows that $F$ is rational if and only if $G$ is, that is, (a)$\Leftrightarrow$(b). By Theorem~\ref{T:FHJZ1}, if (a) or (b) holds, then $\bar T=T-\lambda_0\mathbf{1}$ has a nontrivial invariant subspace, and hence so does $T$.

If $x_1=0$, then $F=FG$; since $G(0)=0$, the series $1-G$ is invertible, so $F\equiv0$ and (a) holds automatically, while (b) need not. In this case $x_1=\langle f,\mathbf{1}\rangle\langle \mathbf{1},g\rangle=0$, so either $\langle \mathbf{1},g\rangle=0$, in which case $\bar T\mathbf{1}=0$ and $\mathbb C\mathbf{1}$ is a nontrivial invariant subspace of $\bar T$, or $\langle f,\mathbf{1}\rangle=0$, in which case $\bar T^*\mathbf{1}=0$ and $\overline{\ran\bar T}$ is one. In both cases $\bar T$, and hence $T$, has a nontrivial invariant subspace.
\end{proof}

An immediate consequence of Theorem \ref{T:H2 perturbation} is the following.

\begin{cor}
Let $T=D+f\otimes g$ with $D=\sum_{n=0}^\infty \lambda_n z^n\otimes z^n\in B(H^2)$ and $f,g\in H^2$, and let $\bar D=D-\lambda_0 I$. If there exists $N\ge 1$ such that $\langle \bar D^n f,g\rangle=0$ for all $n\ge N$, then $T$ has a nontrivial invariant subspace.
\end{cor}

\begin{proof}
Under the hypothesis, the generating function in condition (b) of Theorem \ref{T:H2 perturbation} is a polynomial, hence rational; the conclusion therefore follows from that theorem.
\end{proof}

We conclude this section with a finite-rank analogue of these results.

\begin{thm}\label{T:finite rank perturbations}
Let $N\in\mathbb N$, let $f_i,g_i\in H^2$ for $1\le i\le N$, let $\{e_n\}_{n=0}^{\infty}$ be the standard orthonormal basis of $H^2$, and let
\[
T=D+\sum_{i=1}^N f_i\otimes g_i\in B(H^2),\qquad D:=\sum_{n=1}^\infty \lambda_n e_n\otimes e_n.
\]
If $\langle D^n f_i,g_j\rangle=0$ for all $n\ge 1$ and all $i,j$, and if $\langle f_i,g_j\rangle=0$ for $i\neq j$, then $T$ is intransitive.
\end{thm}
\begin{proof}
Note that $D\mathbf{1}=0$ and $D^*\mathbf{1}=0$, so
\[
\langle T\mathbf{1},\mathbf{1}\rangle=\sum_{i=1}^N \langle f_i,\mathbf{1}\rangle \langle \mathbf{1},g_i\rangle.
\]
\[
T^2=D^2+\sum_{i=1}^N Df_i\otimes g_i+\sum_{i=1}^N f_i\otimes D^*g_i+\sum_{i=1}^N \langle f_i,g_i\rangle f_i\otimes g_i.
\]
\[
\langle T^2\mathbf{1},\mathbf{1}\rangle=\sum_{i=1}^N \langle f_i,g_i\rangle \langle f_i,\mathbf{1}\rangle \langle \mathbf{1},g_i\rangle.
\]
\begin{align*}
T^3=D^3&+\sum_{i=1}^N D^2f_i\otimes g_i+\sum_{i=1}^N Df_i\otimes D^*g_i+\sum_{i=1}^N\langle f_i,g_i\rangle Df_i\otimes g_i\\
&+\sum_{i=1}^N f_i\otimes D^{*2}g_i+\sum_{i=1}^N \langle f_i,g_i\rangle f_i\otimes D^*g_i+\sum_{i=1}^N \langle f_i,g_i\rangle^2 f_i\otimes g_i.
\end{align*}
\[
\langle T^3\mathbf{1},\mathbf{1}\rangle=\sum_{i=1}^N \langle f_i,g_i\rangle^2 \langle f_i,\mathbf{1}\rangle \langle \mathbf{1},g_i\rangle.
\]
A straightforward induction gives
\[
\langle T^n\mathbf{1},\mathbf{1}\rangle=\sum_{i=1}^N \langle f_i,g_i\rangle^{n-1} \langle f_i,\mathbf{1}\rangle \langle \mathbf{1},g_i\rangle.
\]
Evidently, $\sum_{n=0}^\infty \langle T^n\mathbf{1},\mathbf{1}\rangle z^n$ is a rational function.
Thus by Theorem \ref{T:FHJZ1}, $T$ is intransitive.
\end{proof}

\section{On Halmos's third problem}\label{S:Halmos}

Halmos's third problem \cite{Hal70} is one of the ten open problems posed by Halmos in the theory of operators on Hilbert space: is the inverse of an invertible intransitive operator necessarily also intransitive? Equivalently, is the property of having a nontrivial closed invariant subspace preserved under taking inverses? This is natural to ask because $T$ and $T^{-1}$ commute and share many spectral properties, although their invariant subspace lattices need not coincide.

The problem is closely related to the Invariant Subspace Problem, and we refer to the survey \cite{LomShu18} for a comprehensive account. While the other nine of Halmos's ten problems have been settled, the third has remained open for over fifty years; partial positive results under additional assumptions are due to Douglas, Foias and Pearcy \cite{DFP05} and to Jiang, Fang and Ji \cite{JFJ25}, and the companion problem, whether the intransitivity of $T^{2}$ implies that of $T$, was partially answered in the first paper of this series \cite{FHJZ1}. In this section we reduce Halmos's third problem to a special class of operators, namely those of \emph{conjugate bi-geometric form}, whose definition rests on the Cowen-Douglas class.

Recall that, for a connected open set $\Omega\subset\mathbb C$ and a positive integer $n$, an operator $T\in B(\H)$ is said to belong to the Cowen-Douglas class $B_n(\Omega)$ if $\Omega\subset\sigma(T)$; $\ran(T-z)=\H$ for every $z\in\Omega$; $\bigvee_{z\in\Omega}\ker(T-z)=\H$; and $\dim\ker(T-z)=n$ for every $z\in\Omega$. The theory of these operators connects the similarity classification of operators with the complex geometry of holomorphic vector bundles (see \cite{CD} and the book \cite{JW2}), and it is the basis of the reduction established below.

\begin{defn}
	Let $T\in B(\H)$. We say that $T$ has a conjugate bi-geometric form if there is a nontrivial subspace $\mathfrak{M}$ such that
	\[
	T=\begin{matrix}\begin{bmatrix}T_1& T_{12}\\ 0& T_2 \\ \end{bmatrix}&\begin{matrix}
		\mathfrak{M}\\
		\ \ \mathfrak{M}^{\perp}
		\end{matrix}
\end{matrix} \,
	\]
	 where $T_1^*,T_2\in B_1(O_\delta)$ for some $\delta>0$, with $O_\delta:=\{z\in\mathbb{C}:|z|<\delta\}$.
\end{defn}

We shall show that, for Halmos's third problem, it suffices to consider invertible operators with conjugate bi-geometric form (Theorem \ref{C}).

\begin{lem}\label{L:the cap of T^nH}
Let $T\in B(\H)$ and $\mathcal{R}=\bigcap\limits_{n=1}^{\infty}T^n(\H)$. If $T$ is injective, then $T(\mathcal{R})=\mathcal{R}$.
\end{lem}

\begin{proof}
It is evident that $T(\mathcal{R})\subseteq \mathcal{R}$. Let $y\in \mathcal{R}$. Since $y\in T(\H)$, there is $x\in \H$ such that $Tx=y$. For any $n\in \mathbb{N}$, it follows from $y\in T^{n+1}(\H)$ that there is $x_n\in \H$ such that $T^{n+1}x_n=y$. Since $T$ is injective, it follows from $T(T^{n}x_n)=y$ that $T^nx_n=x$.
Thus $x\in\bigcap\limits_{n=1}^{\infty}T^n(\H)=\mathcal{R}$. This completes the proof.
\end{proof}

\begin{lem}\label{L:CD operator}
Let $k\in\mathbb{N}$ and let $T\in B(\mathcal{H})$ be an injective operator with closed range.
Suppose that $\dim\ran(T)^\perp = k$ and
$\bigcap_{n=1}^\infty\ran(T^n) = \{0\}$.
Then $T^*$ is a Cowen-Douglas operator of index $k$ on some open disc.
\end{lem}

\begin{proof}
Since $T$ is injective with closed range, $T$ is left invertible; choose $S\in B(\mathcal{H})$ with $ST=I$, for instance $S=(T^*T)^{-1}T^*$ on $\ran T$ and $S=0$ on $\ran(T)^\perp$. Then
\[
T - \lambda I = (I - \lambda S)T,\qquad \lambda\in\mathbb{C}.
\]
Let $\Omega = \{\lambda\in\mathbb{C} : |\lambda| < 1/\|S\|\}$. For $\lambda\in\Omega$ the operator $I - \lambda S$ is invertible, so $T-\lambda I$ is injective, its range $(I - \lambda S)\ran(T)$ is closed, and
\[
\dim\ran(T - \lambda I)^\perp = \dim\ran(T)^\perp = k.
\]

Set $\Omega^* = \{\overline{\lambda} : \lambda\in\Omega\}$. For $w\in\Omega^*$, the operator $T^* - w$, being the adjoint of $T-\overline{w}$, is surjective and
\[
\dim\ker(T^* - w) = \dim\ran(T - \overline{w})^\perp = k.
\]
It remains to show that $\bigvee_{w\in\Omega^*}\ker(T^* - w)=\mathcal{H}$. By \cite[Proposition 1.41]{JW1}, it suffices to show that $\bigvee_{n\in \mathbb{N}}\ker (T^*)^n=\mathcal{H}$; but this follows at once from the hypothesis, since
\[\bigvee_{n\in \mathbb{N}}\ker (T^*)^n=(\bigcap_{n=1}^\infty\ran(T^n))^\perp=\{0\}^\perp=\mathcal{H}.\]
Hence $T^*$ satisfies all the requirements of the definition, that is, $T^*$ is a Cowen-Douglas operator of index $k$ on the disc $\Omega$.
\end{proof}

\begin{lem}\label{lrinv}
	Let $T\in B(\H)$ be an invertible operator and $\mathfrak{M}$ be a nonzero invariant subspace of $T$, with $P_{\mathfrak{M}}$ the orthogonal projection onto $\mathfrak{M}$. Write $T_1=T|_{\mathfrak{M}}=P_{\mathfrak{M}}TP_{\mathfrak{M}}$ and $T_2=P_{\mathfrak{M}^\perp}TP_{\mathfrak{M}^\perp}$. Then $T_1$ is left invertible and $T_2$ is right invertible.
\end{lem}
\begin{proof}
With respect to the decomposition $\H=\mathfrak{M}\oplus\mathfrak{M}^{\perp}$, write
	\[
T=\begin{matrix}\begin{bmatrix}T_1& T_{12}\\ 0& T_2 \\ \end{bmatrix}&\begin{matrix}
		\mathfrak{M}\\
		\ \ \mathfrak{M}^{\perp}
	\end{matrix}
\end{matrix} \ \ ,
\]
and
	\[
T^{-1}=\begin{matrix}\begin{bmatrix}A_1& A_{12}\\ A_{21}& A_2 \\ \end{bmatrix}&\begin{matrix}
		\mathfrak{M}\\
		\ \ \mathfrak{M}^{\perp}
	\end{matrix}
\end{matrix} \ \ .
\]
Since
\begin{align*}
&T^{-1}T=\begin{bmatrix}A_1& A_{12}\\ A_{21}& A_2 \\ \end{bmatrix}\begin{bmatrix}T_1& T_{12}\\ 0& T_2 \\ \end{bmatrix}=\begin{bmatrix}A_1T_1& A_1T_{12}+A_{12}T_2\\ A_{21}T_1& A_{21}T_{12}+A_2T_2 \\ \end{bmatrix}=\begin{bmatrix}I& 0\\ 0& I \\ \end{bmatrix}, \\
&TT^{-1}=\begin{bmatrix}T_1& T_{12}\\ 0& T_2 \\ \end{bmatrix}\begin{bmatrix}A_1& A_{12}\\ A_{21}& A_2 \\ \end{bmatrix}=\begin{bmatrix}T_1A_1+T_{12}A_{21}& T_1A_{12}+T_{12}A_2\\ T_2A_{21}& T_2A_2 \\ \end{bmatrix}=\begin{bmatrix}I& 0\\ 0& I \\ \end{bmatrix},
\end{align*}
we obtain $A_1T_1=I$ and $T_2A_2=I$; hence $T_1$ is left invertible and $T_2$ is right invertible.
\end{proof}

\begin{lem}\label{intersection}
	Let $T\in B(\H)$ be an invertible operator, let $\mathfrak{M}$ be a nontrivial invariant subspace of $T$, and let $P_{\mathfrak{M}}$ be the orthogonal projection onto $\mathfrak{M}$. If $T|_{\mathfrak{M}}=P_{\mathfrak{M}}TP_{\mathfrak{M}}$ is invertible, then $T^{-1}$ is intransitive. The same conclusion holds if $\bigcap_{n\geq 1}(\ran T|_{\mathfrak{M}}^n)\neq \{0\}$.
\end{lem}
\begin{proof}
	By Lemma \ref{lrinv}, $T|_{\mathfrak{M}}$ is left invertible, that is, injective with closed range.

	If $T|_{\mathfrak{M}}=P_{\mathfrak{M}}TP_{\mathfrak{M}}$ is invertible, then $P_{\mathfrak{M}^{\perp}}T^{-1}P_{\mathfrak{M}}=0$. That is, $\mathfrak{M}\in \textup{Lat}T^{-1}$, and hence $T^{-1}$ is intransitive. In particular, if $\mathfrak{M}$ is finite dimensional, then $P_{\mathfrak{M}}TP_{\mathfrak{M}}$, being left invertible, is automatically invertible.

	Assume now that $\mathcal{R}:=\bigcap_{n\geq 1}(\ran T|_{\mathfrak{M}}^n)\neq \{0\}$. Since $T^n|_{\mathfrak{M}}$ is bounded below, each $T^n(\mathfrak{M})$ is closed, and hence so is $\mathcal{R}$. As $T|_{\mathfrak{M}}\in B(\mathfrak{M})$ is injective, Lemma \ref{L:the cap of T^nH} gives $T|_{\mathfrak{M}}\mathcal{R}=\mathcal{R}$; in particular, $\mathcal{R}$ is a nonzero subspace of $\mathfrak{M}$, hence a nontrivial invariant subspace of $T$. Therefore $T\mathcal{R}=\mathcal{R}$, so that $T^{-1}\mathcal{R}=\mathcal{R}$ and $T^{-1}$ is intransitive.
\end{proof}

We are now ready to prove Theorem \ref{C}.

\begin{proof}[\bf Proof of Theorem \ref{C}]
Since a conjugate bi-geometric operator has, by definition, a nontrivial invariant subspace, $(1)\Rightarrow (2)$ is immediate.

\smallskip

We prove $(2)\Rightarrow (1)$. Let $T\in B(\H)$ be an invertible operator with a nontrivial invariant subspace $\mathcal{N}$. If $\bigcap_{n\geq 1}(\ran T|_{\mathcal{N}}^n)\neq \{0\}$, then Lemma \ref{intersection} shows that $T^{-1}$ is intransitive; so we may suppose that $\bigcap_{n\geq 1}(\ran T|_{\mathcal{N}}^n)= \{0\}$.
In particular, $T|_{\mathcal{N}}$ is not invertible, while Lemma \ref{lrinv} shows that it is left invertible. Hence $\ker (T|_{\mathcal{N}})^*\neq \{0\}$. Choose a unit vector $e_0\in\ker(T|_{\mathcal{N}})^*$ and set
\[
	\mathfrak{M}=\bigvee\{T|_{\mathcal{N}}^ne_0; n\in\mathbb{Z}_{\geq 0}\}.
\]
Then $\mathfrak{M}\subseteq\mathcal{N}$ is a nontrivial invariant subspace of $T$. Write
\[
T=\begin{matrix}\begin{bmatrix}T_1& T_{12}\\ 0& T_2 \\ \end{bmatrix}&\begin{matrix}
		\mathfrak{M}\\
		\ \ \mathfrak{M}^{\perp}
	\end{matrix}
\end{matrix} \ \ .
\]
The restriction $T_1=T|_{\mathfrak{M}}$ is left invertible and $\ran T_1=\mathfrak{M}\ominus(\bigvee \{e_0\})$. Hence $\operatorname{codim}\, \ran T_1=1$, and
\[\bigcap_{n=1}^\infty \ran T_1^n=\bigcap_{n=1}^\infty T^n(\mathfrak{M})\subseteq\bigcap_{n=1}^\infty T^n(\mathcal{N})=\{0\}.\]
By Lemma \ref{L:CD operator}, $T_1^*\in B_1(O_{\delta_1})$ for some $\delta_1>0$.

If $\bigcap_{n\geq 1}(\ran(T_2^*)^n)\neq \{0\}$, then $T^{-1}$ is intransitive. Indeed, with respect to $\H=\mathfrak{M}^{\perp}\oplus\mathfrak{M}$ we may write
\[
T^*=\begin{matrix}\begin{bmatrix}T_2^*& T^*_{12}\\ 0& T_1^* \\ \end{bmatrix}&\begin{matrix}
		\mathfrak{M}^{\perp}\\
		\ \ \mathfrak{M}
	\end{matrix}
\end{matrix} \ \ ,
\]
so that $\mathfrak{M}^{\perp}$ is invariant under $T^*$ and $T^*|_{\mathfrak{M}^{\perp}}=T_2^*$. Lemma \ref{intersection} shows that $(T^*)^{-1}$ is intransitive, and hence so is $T^{-1}$: a nontrivial invariant subspace of $(T^{-1})^*$ has an orthogonal complement invariant under $T^{-1}$.
Hence we may assume that $\bigcap_{n\geq 1}(\ran(T_2^*)^n)=\{0\}$. In particular, $T_2$ is not invertible; since $T_2$ is always right invertible, we have $\ker T_2\neq \{0\}$.

\smallskip

We first observe that $T_{12}f\notin \ran T_1$ for every nonzero $f\in \ker T_2$. Indeed, if $0\neq f\in \ker T_2$ and $T_1 u=T_{12}f$ for some $u\in \mathfrak{M}$, then
\[
T\begin{bmatrix}-u\\ f \end{bmatrix}=\begin{bmatrix}-T_1u+T_{12}f\\T_2f  \end{bmatrix}=\begin{bmatrix}0\\ 0 \end{bmatrix},
\]
which contradicts the invertibility of $T$.

We next show that $\dim \ker T_2=1$. Suppose that $f_0,f_1\in \ker T_2$ are linearly independent. By the observation above, we may write
\[
T_{12}f_0=a_0e_0+g_0 \ \ \  \text{and} \ \ \ T_{12}f_1=a_1e_0+g_1,
\]
where $0\neq a_0,a_1\in\mathbb{C}$ and $g_0, g_1\in \ran T_1$. Then $0\neq a_1f_0-a_0f_1\in\ker T_2$, while
\[
T_{12}(a_1f_0-a_0f_1)=a_1g_0-a_0g_1\in\ran T_1,
\]
which contradicts the observation above.

\smallskip
Now $T_2^*$ is left invertible and $\dim\ker T_2=1$, while by assumption the intersection $\bigcap_{n\geq 1}\ran (T_2^*)^n$ is trivial. Applying Lemma \ref{L:CD operator} to $T_2^*$, we obtain $T_2\in B_1(O_{\delta_2})$ for some $\delta_2>0$.

Taking $\delta=\min\{\delta_1, \delta_2\}$, we conclude that $T$ has a conjugate bi-geometric form. This proves $(2)\Rightarrow(1)$ and completes the proof.
\end{proof}

\begin{thm}\label{T:characterization of invertibility of T}
Let $T\in B(\mathcal{H}\oplus \mathcal{H})$ be a conjugate bi-geometric operator given by
\[
T=\begin{bmatrix}
A & T_{12}\\
0 & B
\end{bmatrix}.
\]
Then the following are equivalent:
\begin{enumerate}
\item $T$ is invertible;
\item $T_{12}(\ker B) \not\subset \ran A$;
\item $T_{12}^*(\ker A^*) \not\subset \ran B^*$.
\end{enumerate}
\end{thm}

\begin{proof}
Since $A^*\in B_1(O_\delta)$ and $0\in O_\delta$, the operator $A^*$ is surjective; equivalently, $A$ is injective with closed range, and $\dim\ker A^*=1$. Choose $0\ne \varphi\in \ker A^*$. Then
\[
\ran A =  (\ker A^*)^\perp= \{x\in \mathcal{H}:\langle x,\varphi\rangle=0\}.
\]
Similarly, $B$ is surjective and $\dim\ker B=1$; choose $0\ne \psi\in \ker B$.

(1) $\Rightarrow$ (2). Suppose $T$ is invertible. If $T_{12}\psi\in \ran A$, then there exists $x\in \mathcal{H}$ with $Ax=T_{12}\psi$. Hence
\[
T\begin{bmatrix}-x\\ \psi\end{bmatrix}
=
\begin{bmatrix}
-Ax+T_{12}\psi\\
B\psi
\end{bmatrix}
=
\begin{bmatrix}0\\0\end{bmatrix},
\]
contradicting injectivity of $T$. Thus $T_{12}\psi\notin \ran A$.

\smallskip

(2) $\Rightarrow$ (1). Assume $T_{12}\psi\notin \ran A$. We first show that $T$ is injective. Suppose
\[T\begin{bmatrix}x\\ y\end{bmatrix}
=0.\]
Then $By=0$, so $y=\alpha \psi$ for some scalar $\alpha$. Also $Ax+T_{12}y=0$. If $\alpha\ne0$, then
\[
A(-\alpha^{-1}x)=T_{12}\psi,
\]
so $T_{12}\psi\in \ran A$, a contradiction. Hence $\alpha=0$, so $y=0$. Then $Ax=0$, and injectivity of $A$ gives $x=0$. Thus $T$ is injective.

\smallskip

We next show that $T$ is surjective. Let $(u,v)\in \mathcal{H}\oplus \mathcal{H}$. Since $B$ is surjective, choose $y_0\in \mathcal{H}$ with $By_0=v$. The general solution of $By=v$ is $y=y_0+\alpha\psi$, $\alpha\in\mathbb C$. We need $x\in \mathcal{H}$ such that
\[
Ax+T_{12}(y_0+\alpha\psi)=u.
\]
This is possible if and only if
\[
u-T_{12}(y_0+\alpha\psi)\in \ran A = (\ker A^*)^\perp.
\]
Equivalently, for $\varphi\in\ker A^*$,
\[
\langle u-T_{12}y_0-\alpha T_{12}\psi,\varphi\rangle=0.
\]
Since $T_{12}\psi\notin \ran A$, we have $\langle T_{12}\psi,\varphi\rangle\ne0$. Thus we may choose uniquely
\[
\alpha=\frac{\langle u-T_{12}y_0,\varphi\rangle}{\langle T_{12}\psi,\varphi\rangle}.
\]
With this choice, there exists $x\in \mathcal{H}$ satisfying the equation, so
\[
T\begin{bmatrix}x\\ y_0+\alpha\psi\end{bmatrix}=\begin{bmatrix}u\\ v\end{bmatrix}.
\]
Hence $T$ is surjective. By the bounded inverse theorem, $T$ is invertible.

\smallskip

(2) $\Leftrightarrow$ (3).
Note that
\[
\ran A = (\ker A^*)^\perp = \{\varphi\}^\perp,
\]
and
\[
\ran B^* = (\ker B)^\perp = \{\psi\}^\perp.
\]

First,
\[
T_{12}(\ker B) \not\subset \ran A
\iff
T_{12}\psi \notin \ran A
\iff
\langle T_{12}\psi,\varphi\rangle \ne 0. \tag{*}
\]

Second,
\[
T_{12}^*(\ker A^*) \not\subset \ran B^*
\iff
T_{12}^*\varphi \notin \ran B^*
\iff
\langle T_{12}^*\varphi,\psi\rangle \ne 0. \tag{**}
\]

But
\[
\langle T_{12}^*\varphi,\psi\rangle
= \langle \varphi, T_{12}\psi \rangle
= \overline{\langle T_{12}\psi,\varphi\rangle}.
\]
Hence
\[
\langle T_{12}\psi,\varphi\rangle \ne 0
\iff
\langle T_{12}^*\varphi,\psi\rangle \ne 0.
\]
Combining this with (*) and (**) yields
\[
T_{12}(\ker B) \not\subset \ran A
\iff
T_{12}^*(\ker A^*) \not\subset \ran B^*.
\]

\end{proof}

\smallskip

\begin{lem}\label{L:left inverse}
Let $A\in B(\mathcal{H})$ with $A^*\in B_1(O_\delta)$ for some $\delta>0$. Suppose that $x\in \ker A^*$ and $y\in \mathcal{H}$ satisfy $\langle y,x\rangle=1$.
Then there is a unique left inverse $A_L^{-1}$ of $A$ such that $AA_{L}^{-1}=I-y\otimes x$. Moreover, $A_L^{-1}y=0$.
\end{lem}

\begin{proof}
Since $\langle y,x\rangle=1$, the vectors $x$ and $y$ are nonzero.
Let
\[x_0=\frac{x}{\|x\|},~~y_0=\|x\|y.\]
Then $\|x_0\|=1$, and $\langle y_0,x_0\rangle=1$.
Write $y_0= x_0 + v$, where $\langle x_0, v \rangle = 0$, i.e., $v \in \ran A$, and choose $z_0$ such that $Az_0 = v$.
Set \[A_L^{-1} = A_{SL}^{-1} - z_0 \otimes x_0,\] where $A_{SL}^{-1}$ is the standard left inverse of $A$.
Then
\[
A_L^{-1}A = A_{SL}^{-1}A - (z_0 \otimes x_0)A = I,
\]
and
\[
A A_L^{-1} = A A_{SL}^{-1} - A z_0 \otimes x_0 = I - x_0 \otimes x_0 - v \otimes x_0 = I - y_0 \otimes x_0=I-y\otimes x.
\]
Moreover,
\[
A A_L^{-1}y=(I-y\otimes x)y=y-\langle y,x\rangle y=0.
\]
Since $A^*$ is surjective, $A$ is injective, so $A_L^{-1}y=0$. Finally, $A_L^{-1}$ is the unique left inverse of $A$ with the required property, because $A$ is injective.
\end{proof}

By Theorem \ref{T:characterization of invertibility of T}, we can write the form of $T^{-1}$ explicitly if $T$ is invertible.

\begin{prop}\label{P:form of inverse}
Let $T\in B(\mathcal{H}\oplus \mathcal{H})$ be an invertible conjugate bi-geometric operator given by
\[
T=\begin{bmatrix}
A & T_{12}\\
0 & B
\end{bmatrix}.
\]
Choose nonzero vectors $e_0\in \ker A^*$ and $f_0\in \ker B$ such that $\langle T_{12}f_0,e_0 \rangle=1$. Let $A_L^{-1}$
be the unique left inverse of $A$ such that $A A_L^{-1} = I - T_{12}f_0 \otimes e_0$;
and $B_R^{-1}$ be the unique right inverse of $B$ such that $ B_R^{-1}B = I -f_0 \otimes T_{12}^*e_0$.
Then
\[T^{-1}=\begin{bmatrix} A_L^{-1} & -A_L^{-1}T_{12}B_R^{-1} \\ f_0 \otimes e_0 & B_R^{-1} \end{bmatrix}.\]
\end{prop}

\begin{proof}
Since $T$ is invertible, Theorem \ref{T:characterization of invertibility of T} shows that such vectors $e_0$ and $f_0$ exist. Now $A^*\in B_1(O_\delta)$ for some $\delta>0$, while $e_0\in \ker A^*$ and $\langle T_{12}f_0,e_0\rangle=1$; by Lemma \ref{L:left inverse}, there is a unique left inverse $A_L^{-1}$ of $A$ such that
\[A A_L^{-1} = I - T_{12}f_0 \otimes e_0,\]
and
\[A_L^{-1}T_{12}f_0=0.\]

Similarly, $(B^*)^*=B\in B_1(O_\delta)$, $f_0\in \ker B=\ker (B^*)^*$, and $\langle T_{12}^*e_0, f_0\rangle=\overline{\langle T_{12}f_0,e_0\rangle}=1$. By Lemma \ref{L:left inverse} again, there is a unique left inverse $C$ of $B^*$ such that
  \[B^*C=I-T_{12}^*e_0\otimes f_0,\]
  and
  \[C(T_{12}^*e_0)=0.\]
Set $B_R^{-1}=C^*$. Then
\[
B_R^{-1}B =(B^*C)^*=(I-T_{12}^*e_0\otimes f_0)^*=I - f_0 \otimes T_{12}^* e_0, \]
so that $B_R^{-1}$ is a right inverse of $B$, and
\[ (B_R^{-1})^*(T_{12}^* e_0)=C(T_{12}^*e_0)=0.
\]

Write
\[
Q = \begin{bmatrix} A_L^{-1} & -A_L^{-1}T_{12}B_R^{-1} \\ f_0 \otimes e_0 & B_R^{-1} \end{bmatrix}.
\]
Then,
\begin{align*}
TQ &= \begin{bmatrix} A & T_{12} \\ 0 & B \end{bmatrix}
\begin{bmatrix} A_L^{-1} & -A_L^{-1}T_{12}B_R^{-1} \\ f_0 \otimes e_0 & B_R^{-1} \end{bmatrix} \\
&= \begin{bmatrix} A A_L^{-1} + T_{12}f_0 \otimes e_0 & -A A_L^{-1}T_{12}B_R^{-1} + T_{12}B_R^{-1} \\ B f_0 \otimes e_0 & B B_R^{-1} \end{bmatrix} \\
&= \begin{bmatrix} I & (T_{12}f_0 \otimes e_0)T_{12}B_R^{-1} \\ 0 & I \end{bmatrix}\\
&= \begin{bmatrix} I & 0 \\ 0 & I \end{bmatrix},
\end{align*}
where the $(1,2)$ entry vanishes because $(B_R^{-1})^*(T_{12}^*e_0)=0$.

\begin{align*}
QT  &=
\begin{bmatrix} A_L^{-1} & -A_L^{-1}T_{12}B_R^{-1} \\ f_0 \otimes e_0 & B_R^{-1} \end{bmatrix}\begin{bmatrix} A & T_{12} \\ 0 & B \end{bmatrix} \\
&= \begin{bmatrix} A_L^{-1}A & A_L^{-1}T_{12} - A_L^{-1}T_{12}B_R^{-1}B \\ (f_0 \otimes e_0)A & (f_0 \otimes e_0)T_{12} + B_R^{-1}B \end{bmatrix} \\
&= \begin{bmatrix} I & A_L^{-1}T_{12}(f_0 \otimes T_{12}^* e_0) \\ f_0 \otimes (A^* e_0) & I \end{bmatrix}\\
&= \begin{bmatrix} I & 0 \\ 0 & I \end{bmatrix},
\end{align*}
where the $(1,2)$ entry vanishes because $A_L^{-1}T_{12}f_0=0$, and the $(2,1)$ entry vanishes because $e_0\in\ker A^*$.

Therefore $T$ is invertible with $T^{-1} = Q$.
\end{proof}

\begin{cor}\label{C:rank-one}
Let $T\in B(\mathcal{H}\oplus \mathcal{H})$ be a conjugate bi-geometric operator given by
\[
T=\begin{bmatrix}
A & T_{12}\\
0 & B
\end{bmatrix},
\]
where $T_{12}=e\otimes f$ is a rank-one operator. Then $T$ is invertible if and only if $e\notin \ran A$ and $f\notin \ran B^*$.
Furthermore, if $T$ is invertible, then $T^{-1}$ is intransitive.
\end{cor}

\begin{proof}
Choose unit vectors $e_1\in \ker A^*$ and $f_1\in \ker B$. By Theorem \ref{T:characterization of invertibility of T}, $T$ is invertible if and only if
$T_{12}f_1\notin \ran A$ and $T_{12}^*e_1\notin \ran B^*$. Therefore, $T$ being invertible is equivalent to
\[\langle f_1,f\rangle e\notin \ran A,~~\langle e_1,e \rangle f \notin \ran B^*.\]
Hence these are equivalent to
\[\langle f_1,f\rangle\neq 0, e\notin \ran A,\langle e_1,e \rangle\neq 0, f \notin \ran B^*.\]
Note that $\langle f_1,f\rangle\neq 0$ and $f \notin \ran B^*$ are equivalent;
$e\notin \ran A$ and $\langle e_1,e \rangle\neq 0$ are equivalent. Therefore,
$T$ is invertible if and only if $e\notin \ran A$ and $f\notin \ran B^*$.

Furthermore, suppose that $T$ is invertible, and choose nonzero vectors $e_0\in \ker A^*$ and $f_0\in \ker B$ such that $\langle T_{12}f_0,e_0 \rangle=1$. Let $A_L^{-1}$ be the unique left inverse of $A$ such that $A A_L^{-1} = I - T_{12}f_0 \otimes e_0$, and let $B_R^{-1}$ be the unique right inverse of $B$ such that $ B_R^{-1}B = I -f_0 \otimes T_{12}^*e_0$. By Proposition \ref{P:form of inverse},
\[T^{-1}=\begin{bmatrix} A_L^{-1} & -A_L^{-1}T_{12}B_R^{-1} \\ f_0 \otimes e_0 & B_R^{-1} \end{bmatrix}=\begin{bmatrix} A_L^{-1} & -(A_L^{-1}e)\otimes ((B_R^{-1})^*f) \\ f_0 \otimes e_0 & B_R^{-1} \end{bmatrix}.\]
Hence,
\[
A A_L^{-1}T_{12}f_0=(I-T_{12}f_0\otimes e_0)T_{12}f_0=T_{12}f_0-\langle T_{12}f_0,e_0\rangle T_{12}f_0=0.
\]
As $A$ is injective, we have $A_L^{-1}T_{12}f_0=0$.
Thus,
\[A_L^{-1}(e\otimes f)f_0=0.\]
That is, $\langle f_0, f \rangle A_L^{-1}e=0$. Since $f_0\in\ker B$ and $f\notin \ran B^*=(\ker B)^\perp$, we have $\langle f_0, f \rangle\neq 0$.
Therefore, $A_{L}^{-1}e=0$. Consequently,
\[T^{-1}=\begin{bmatrix} A_L^{-1} & -(A_L^{-1}e)\otimes ((B_R^{-1})^*f) \\ f_0 \otimes e_0 & B_R^{-1} \end{bmatrix}=\begin{bmatrix} A_L^{-1} &  0\\ f_0 \otimes e_0 & B_R^{-1} \end{bmatrix}.\]
Since the $(1,2)$ entry of $T^{-1}$ vanishes, $\{0\}\oplus\mathcal{H}$ is a nontrivial invariant subspace of $T^{-1}$, so that $T^{-1}$ is intransitive.

\end{proof}

\begin{rem}
The operator $T$ in the above corollary is a generalization of the bilateral shift operator.
\end{rem}

We now prove the general criterion stated in the introduction as Theorem \ref{D}.

\begin{proof}[\bf Proof of Theorem \ref{D}]
Let $\mathcal{N}=\ran P$, so that $\H=\mathcal{N}\oplus\mathcal{N}^{\perp}$ and
\[
T=\begin{bmatrix}A&T_{12}\\0&B\end{bmatrix},\qquad T_{12}=PT(I-P)=e\otimes f,
\]
where $e\in\mathcal{N}$ and $f\in\mathcal{N}^{\perp}$ are nonzero, since $PT(I-P)$ is of rank one. As $T$ is invertible, Lemma \ref{lrinv} shows that $A$ is left invertible and $B$ is right invertible.

If $\bigcap_{n\geq 1}\ran A^n\neq\{0\}$, then Lemma \ref{intersection} shows that $T^{-1}$ is intransitive. If $\bigcap_{n\geq 1}\ran((B^*)^n)\neq\{0\}$, then the same lemma, applied to the invertible operator $T^*$ and its invariant subspace $\mathcal{N}^{\perp}$, gives the same conclusion. Thus we may assume that
\[
\bigcap_{n\geq 1}\ran A^n=\bigcap_{n\geq 1}\ran((B^*)^n)=\{0\};
\]
in particular, neither $A$ nor $B$ is invertible, so that $\ker A^*\neq\{0\}$ and $\ker B\neq\{0\}$.

\smallskip

We claim that $\dim\ker A^*=\dim\ker B=1$. Let $0\neq y\in\ker B$. If $T_{12}y=Au$ for some $u\in\mathcal{N}$, then
\[
T\begin{bmatrix}-u\\ y\end{bmatrix}=\begin{bmatrix}-Au+T_{12}y\\ By\end{bmatrix}=\begin{bmatrix}0\\ 0\end{bmatrix},
\]
which contradicts the invertibility of $T$; hence $T_{12}y=\langle y,f\rangle e\notin\ran A$, and in particular $\langle y,f\rangle\neq0$. If $g_1,g_2\in\ker B$ were linearly independent, then $y=\langle g_2,f\rangle g_1-\langle g_1,f\rangle g_2$ would be a nonzero vector in $\ker B$ with $\langle y,f\rangle=0$, a contradiction. Hence $\dim\ker B=1$. The same argument, applied to
\[
T^*=\begin{bmatrix}B^*&T_{12}^*\\0&A^*\end{bmatrix}
\]
with respect to $\H=\mathcal{N}^{\perp}\oplus\mathcal{N}$, where $T_{12}^*=f\otimes e$ is again of rank one, yields $\dim\ker A^*=1$.

\smallskip

By Lemma \ref{L:CD operator}, applied to $A$ and to $B^*$, we obtain $A^*\in B_1(O_{\delta_1})$ and $B\in B_1(O_{\delta_2})$ for some $\delta_1,\delta_2>0$. Thus $T$ is an invertible conjugate bi-geometric operator with respect to $\mathcal{N}\oplus\mathcal{N}^{\perp}$ whose off-diagonal entry is of rank one, and Corollary \ref{C:rank-one} completes the proof.
\end{proof}

\begin{eg}
Let $A\in B(H^2)$ be an invertible analytic Toeplitz operator, let $C\in B(H^2)$ be an invertible co-analytic Toeplitz operator, and let $B\in \ran \tau_{S,S^*}$.
Set
	\[
T_1=\begin{matrix}\begin{bmatrix}S& B+AC\\ 0& S^* \\ \end{bmatrix}&\begin{matrix}
		H^2\\
		H^2
	\end{matrix}
\end{matrix} \,
\]
Then $T_1$ is invertible and $T_1^{-1}$ is intransitive.

\noindent{\bf Step 1.} In fact, as $A$ is an invertible analytic Toeplitz operator, it follows that $AS=SA$; similarly, as $C$ is an invertible co-analytic Toeplitz operator, $CS^*=S^*C$. Both assertions are instances of \cite[Theorem 3.2.13]{MarRos07}: the shift and the backward shift are the analytic and co-analytic Toeplitz operators with symbols $z$ and $\bar z$, respectively, and two Toeplitz operators with symbols in $L^\infty$ commute whenever they are both analytic or both co-analytic.

Since $B\in \ran \tau_{S,S^*}$, there exists $X\in B(H^2)$ such that $B=SX-XS^*$.
Define
\[G=\begin{bmatrix}I& -X\\ 0& I\\ \end{bmatrix}\begin{bmatrix}A& 0\\ 0& C^{-1}\\ \end{bmatrix},\]
Then clearly $G$ is invertible.

Then it is straightforward to check that
\[
\begin{bmatrix}A^{-1}& 0\\ 0& C \\ \end{bmatrix}\begin{bmatrix}I& X\\ 0& I\\ \end{bmatrix}\begin{bmatrix}S& B+AC\\ 0& S^* \\ \end{bmatrix}\begin{bmatrix}I& -X\\ 0& I\\ \end{bmatrix}\begin{bmatrix}A& 0\\ 0& C^{-1} \\ \end{bmatrix}=G^{-1}T_1G=\begin{bmatrix}S& I\\ 0& S^* \\ \end{bmatrix}.
\]
Thus $T_1$ is similar to the operator
\[
T=\begin{bmatrix}S& I\\ 0& S^* \\ \end{bmatrix},
\]
and it suffices to show that $T$ is invertible and that $T^{-1}$ is intransitive.

\bigskip

\noindent{\bf Step 2.} Let $\{e_0,e_1,e_2,\cdots\}$ and $\{f_0,f_1,f_2,\cdots\}$ be orthonormal bases of $\mathcal{H}_1:=H^2\oplus 0$ and $\mathcal{H}_2:=0\oplus H^2$, respectively. With respect to the union of these two bases, which is an orthonormal basis of $\mathcal{H}_1\oplus\mathcal{H}_2$, the operator $T$ is given by
	\[
T=\begin{matrix}\begin{bmatrix}S& I\\ 0& S^* \\ \end{bmatrix}&\begin{matrix}
		\mathcal{H}_1\\
		 \mathcal{H}_2
	\end{matrix}
\end{matrix} \,
\]
where $S$ is the forward shift operator and $S^*$ is the backward shift operator. It is evident that $T$ has a conjugate bi-geometric form. By Theorem \ref{T:characterization of invertibility of T} and Proposition \ref{P:form of inverse}, $T$ is invertible and
	\[
T^{-1}=\begin{matrix}\begin{bmatrix}S^*& -I\\ f_0\otimes e_0& S \\ \end{bmatrix}&\begin{matrix}
		\mathcal{H}_1\\
		\mathcal{H}_2
	\end{matrix}
\end{matrix} \,
\]

\noindent{\bf Step 3.} Put
\[
x_0=\sum\limits_{n=0}^{\infty} \alpha_ne_n.
\]
We have
\begin{align*}
T^{-1}x_0&=S^*(\sum\limits_{n=0}^{\infty} \alpha_ne_n)+ f_0\otimes e_0 (\sum\limits_{n=0}^{\infty} \alpha_ne_n) \\
&=\sum\limits_{n=0}^{\infty} \alpha_{n+1}e_n+\alpha_0f_0, \\
T^{-2}x_0&=(\alpha_2-\alpha_0)e_0+\sum\limits_{n=1}^{\infty} \alpha_{n+2}e_n+ \alpha_1f_0+ \alpha_0f_1, \\
T^{-3}x_0&=(\alpha_3-\alpha_1)e_0+(\alpha_4-\alpha_0)e_1+\sum\limits_{n=2}^{\infty} \alpha_{n+3}e_n+(\alpha_2-\alpha_0)f_0+ \alpha_1f_1+ \alpha_0f_2.
\end{align*}
We claim that for $k\ge 3$,
\[T^{-k}x_0=\sum\limits_{n=0}^{\infty} \alpha_{n+k}e_n-\sum\limits_{n=0}^{k-2} \alpha_{k-2-n}e_n+\sum\limits_{n=0}^{k-1} \alpha_{k-1-n}f_n-\sum\limits_{n=0}^{k-3} \alpha_{k-3-n}f_n. \tag{*}\]
It is easy to see that the equation (*) holds for $k=3$. Suppose that the equation (*) holds for $k=m$. Since
\begin{align*}
	T^{-(m+1)}x_0=& T^{-1}(T^{-m}x_0) \\
	=& T^{-1}\sum\limits_{n=0}^{\infty}\alpha_{n+m}e_n-T^{-1}\sum\limits_{n=0}^{m-2} \alpha_{m-2-n}e_n \\
	&+T^{-1}\sum\limits_{n=0}^{m-1} \alpha_{m-1-n}f_n-T^{-1}\sum\limits_{n=0}^{m-3} \alpha_{m-3-n}f_n \\
	=& \left(\sum\limits_{n=0}^{\infty}\alpha_{n+m+1}e_n+\alpha_{m}f_0\right)-\left(\sum\limits_{n=0}^{m-3} \alpha_{m-3-n}e_n+\alpha_{m-2}f_0\right) \\
	&+\left(\sum\limits_{n=0}^{m-1} \alpha_{m-1-n}f_{n+1}-\sum\limits_{n=0}^{m-1}\alpha_{m-1-n}e_n\right) \\
	&-\left(\sum\limits_{n=0}^{m-3} \alpha_{m-3-n}f_{n+1}-\sum\limits_{n=0}^{m-3} \alpha_{m-3-n}e_n\right) \\
	=& \sum\limits_{n=0}^{\infty} \alpha_{n+m+1}e_n-\sum\limits_{n=0}^{m-1} \alpha_{m-1-n}e_n+\sum\limits_{n=0}^{m} \alpha_{m-n}f_n-\sum\limits_{n=0}^{m-2} \alpha_{m-2-n}f_n,
\end{align*}
the equation (*) holds for $k=m+1$. Therefore, the equation (*) holds for all $k\ge 3$. In particular,
\[
\langle x_0,~ e_0 \rangle=\alpha_0, \ \ \langle T^{-1}x_0,~ e_0 \rangle=\alpha_1 \ \ \text{and} \ \ \langle T^{-k}x_0,~ e_0 \rangle=\alpha_k-\alpha_{k-2} \ \ \text{for} \ k\ge 2.
\]
Then the function
\[
g(z):=\sum\limits_{k=0}^{\infty}\langle T^{-k}x_0,~ e_0 \rangle z^k=(1-z^2)\sum\limits_{k=0}^{\infty} \alpha_k z^k.
\]
If we set $\alpha_k=2^{-k}$ for $k\in\mathbb{N}$, then
\[g(z)=\frac{1-z^2}{1-{z}/{2}}\]
 is a rational function. Therefore,
 by Theorem \ref{T:FHJZ1}, the operator $T^{-1}$ has a nontrivial invariant subspace.
\end{eg}
 \begin{rem}
 We could also prove the intransitivity of $T^{-1}$ by Theorem \ref{A}; this example reveals that the Rosenblum operators indeed play an important role in studying the invariant subspace problem once more.
 \end{rem}

\begin{rem}
\emergencystretch=1.5em Very recently, Ji, Jiang and Wang \cite{JJW} have made further progress along this line. Employing the Theorem \ref{C} of the present paper as well as techniques from complex geometry and operator theory, they prove that the inverse of a conjugate bi-geometric operator $T$ has a nontrivial invariant subspace whenever $T$ has the bundle-preserving property (\cite[Theorem 3.5]{JJW}), thus giving an affirmative answer to Halmos's third problem for this class of operators. They also show that the same conclusion holds for a class of essentially unitary conjugate bi-geometric operators associated with weighted Bergman spaces, and they characterize the scope of the operators for which their methods apply.

\end{rem}

\section{Tridiagonal operators}\label{S:tridiagonal}

The study of invariant subspaces of tridiagonal operators goes back to Duren \cite{Duren63}, who showed that for a bounded regular tridiagonal operator on $\ell^2$ the lattice of invariant subspaces is anti-isomorphic to the lattice of inner functions. Grivaux \cite{Gri02} later showed, by means of moment sequences, that the weak symmetry condition $t_{n,n+1}t_{n+1,n}\geq 0$ on a real tridiagonal matrix already forces the operator to have a nontrivial closed invariant subspace. 
Clemente and Gallardo-Guti\'errez \cite{CG26} proved more recently that a hyponormal operator of the form $T=R+V$, with $R$ tridiagonal and $V$ of trace class, has nontrivial closed hyperinvariant subspaces unless it is a multiple of the identity; the key step is that no non-normal hyponormal tridiagonal operator is quasitriangular. Whether every tridiagonal operator has a nontrivial closed invariant subspace nevertheless remains open; the theorem below provides another affirmative case.

\begin{thm}\label{T:tridiagonal}
Let $\{e_0,e_1,e_2,\dots\}$ be an orthonormal basis of $\mathcal H$, and let $T\in B(\mathcal H)$ be the tridiagonal operator whose matrix with respect to this basis is
\[
T=\begin{bmatrix}
\lambda_0 & \beta_1 & & & \\
\alpha_1 & \lambda_1 & \beta_2 & & \\
& \alpha_2 & \lambda_2 & \beta_3 & \\
& & \ddots & \ddots & \ddots
\end{bmatrix}.
\]
Suppose that $\lambda_{2k}=\lambda$ for all $k\ge0$ and that
\[
\frac{\sup_{k\ge1}|\beta_{2k-1}|}{\inf_{k\ge1}|\alpha_{2k}|}<1.
\]
Then $T$ is intransitive. If, in addition,
\[
\frac{\sup_{k\ge1}|\beta_{2k}|}{\inf_{k\ge0}|\alpha_{2k+1}|}<1,
\]
then $T$ admits a sequence $\{\mathfrak M_n\}_{n=0}^{\infty}$ of invariant subspaces such that
\[
\mathfrak M_0\supsetneq \mathfrak M_1\supsetneq \mathfrak M_2\supsetneq\cdots.
\]
\end{thm}

\begin{proof}
\noindent{\bf Step 1.} With respect to the decomposition
\[
\mathcal{H}=(\bigvee\{e_0,e_2,e_4,\dots\})
\oplus
(\bigvee\{e_1,e_3,e_5,\dots\}),
\]
the operator $T$ has the block form
\[
T=
\begin{bmatrix}
D_1 & A\\
B & D_2
\end{bmatrix},
\]
where
\[
D_1=\lambda I,
\qquad
A=
\begin{bmatrix}
\beta_1 & 0 & 0 & \cdots\\
\alpha_2 & \beta_3 & 0 & \cdots\\
0 & \alpha_4 & \beta_5 & \cdots\\
\vdots & \vdots & \vdots & \ddots
\end{bmatrix},
\]
and
\[
B=
\begin{bmatrix}
\alpha_1 & \beta_2 & 0 & \cdots\\
0 & \alpha_3 & \beta_4 & \cdots\\
0 & 0 & \alpha_5 & \cdots\\
\vdots & \vdots & \vdots & \ddots
\end{bmatrix}.
\]
Thus
\[
T-\lambda I=
\begin{bmatrix}
0 & A\\
B & \widetilde D_2
\end{bmatrix},
\]
where $\widetilde D_2=D_2-\lambda I$.

\noindent{\bf Step 2.} Write
\[
A=D_A+V_A,
\]
where $D_A$ is the diagonal part of $A$ and $V_A$ is its lower triangular part. Consider the operator
\[
W_A:\bigvee\{e_0,e_2,e_4,\dots\}\longrightarrow \bigvee\{e_1,e_3,e_5,\dots\}
\]
given, with respect to the bases $(e_0,e_2,e_4,\dots)$ and $(e_1,e_3,e_5,\dots)$, by
\[
W_A=
\begin{bmatrix}
0 & \frac{1}{\alpha_2} & 0 & \cdots\\
0 & 0 & \frac{1}{\alpha_4} & \cdots\\
0 & 0 & 0 & \ddots\\
\vdots & \vdots & \vdots & \ddots
\end{bmatrix}.
\]
Then $W_AV_A=I$. Hence
\[
A=D_A+V_A
=D_AW_AV_A+V_A
=(D_AW_A+I)V_A.
\]
Since, with respect to the basis $(e_0,e_2,e_4,\dots)$,
\[
D_AW_Ae_{2k}=\frac{\beta_{2k-1}}{\alpha_{2k}}e_{2k-2}\quad(k\ge1),
\qquad
D_AW_Ae_0=0,
\]
the operator $D_AW_A$ is a weighted backward shift with weights $\frac{\beta_1}{\alpha_2},\frac{\beta_3}{\alpha_4},\frac{\beta_5}{\alpha_6},\dots$; in particular,
\[
\|D_AW_A\|=\sup_{k\ge1}\frac{|\beta_{2k-1}|}{|\alpha_{2k}|}
\le\frac{\sup_{k\ge1}|\beta_{2k-1}|}{\inf_{k\ge1}|\alpha_{2k}|}<1
\]
by the assumption. Therefore $D_AW_A+I$ is invertible. Since the weights $\alpha_2,\alpha_4,\dots$ of the backward shift $V_A$ are bounded below by $\inf_{k\ge1}|\alpha_{2k}|>0$, the operator $V_A$ is left invertible, while $V_A$ itself is not invertible; as $A=(D_AW_A+I)V_A$ with $D_AW_A+I$ invertible, it follows that the same holds for $A$. Thus $A^*$ is surjective and has a nontrivial kernel.

Consequently the operator
\[
(T-\lambda I)^*=
\begin{bmatrix}
0 & B^*\\
A^* & \widetilde D_2^*
\end{bmatrix}
\]
has a nontrivial kernel: for every $0\neq u_0\in\ker A^*$,
\[
(T-\lambda I)^*
\begin{bmatrix}
u_0\\
0
\end{bmatrix}
=0.
\]
Consequently $\ran(T-\lambda I)$ is not dense in $\mathcal H$. Since $\ran(T-\lambda I)$ is invariant under $T$ and nonzero, its closure is a nontrivial closed invariant subspace of $T$. Hence $T$ is intransitive.

\bigskip

\noindent{\bf Step 3.} Now assume additionally that
\[
\frac{\sup_{k\ge1}|\beta_{2k}|}{\inf_{k\ge0}|\alpha_{2k+1}|}<1.
\]
Write
\[
B^*=D_B^*+V_B^*,
\]
where $D_B^*$ is the diagonal part of $B^*$ and $V_B^*$ is its lower triangular part. Since
\[
B^*=D_B^*\bigl(I+(D_B^*)^{-1}V_B^*\bigr)
\]
and
\[
\|(D_B^*)^{-1}V_B^*\|
\le
\frac{\sup_{k\ge1}|\beta_{2k}|}{\inf_{k\ge0}|\alpha_{2k+1}|}<1,
\]
the operator $I+(D_B^*)^{-1}V_B^*$ is invertible, and so is $B^*$.

Fix $0\neq u_0\in\ker A^*$ and put
\[
x_0=
\begin{bmatrix}
u_0\\
0
\end{bmatrix},
\]
so that $(T-\lambda I)^*x_0=0$.

We construct a sequence $\{x_n\}_{n=0}^{\infty}$ of nonzero vectors such that
\[
(T-\lambda I)^*x_{n+1}=x_n,\qquad n\ge0.
\]
Suppose that $x_n=(u_n,v_n)^T\neq0$ has been chosen, and set
\[
v_{n+1}=(B^*)^{-1}u_n.
\]
Since $A^*$ is surjective, we may choose $u_{n+1}$ satisfying
\[
A^*u_{n+1}=v_n-\widetilde D_2^*v_{n+1}.
\]
If $v_n-\widetilde D_2^*v_{n+1}\neq0$, then necessarily $u_{n+1}\neq0$; if $v_n-\widetilde D_2^*v_{n+1}=0$, choose any nonzero $u_{n+1}\in\ker A^*$. In either case,
\[
x_{n+1}=
\begin{bmatrix}
u_{n+1}\\
v_{n+1}
\end{bmatrix}
\neq0.
\]
Then
\[
(T-\lambda I)^*x_{n+1}
=
\begin{bmatrix}
0 & B^*\\
A^* & \widetilde D_2^*
\end{bmatrix}
\begin{bmatrix}
u_{n+1}\\
v_{n+1}
\end{bmatrix}
=
\begin{bmatrix}
B^*v_{n+1}\\
A^*u_{n+1}+\widetilde D_2^*v_{n+1}
\end{bmatrix}
=
\begin{bmatrix}
u_n\\
v_n
\end{bmatrix}
=x_n.
\]
Thus the required sequence exists.

For $n\ge0$, let
\[
\mathfrak N_n=\bigvee\{x_0,x_1,\dots,x_n\}.
\]
Since
\[
T^*x_{n+1}=\overline{\lambda}x_{n+1}+x_n\quad(n\ge0)
\qquad\text{and}\qquad
T^*x_0=\overline{\lambda}x_0,
\]
each $\mathfrak N_n$ is invariant under $T^*$.

\noindent{\bf Step 4.} We claim that these subspaces form a strictly increasing chain:
\[
\mathfrak N_0\subsetneq \mathfrak N_1\subsetneq \mathfrak N_2\subsetneq\cdots.
\]
Suppose, to the contrary, that $x_{n+1}\in\mathfrak N_n$ for some $n$, say
\[
x_{n+1}=\sum_{k=0}^{n}c_kx_k.
\]
Applying $(T-\lambda I)^*$ and using
\[
(T-\lambda I)^*x_0=0,\qquad
(T-\lambda I)^*x_j=x_{j-1}\quad(j\ge1),
\]
we obtain
\[
x_n=(T-\lambda I)^*x_{n+1}
=\sum_{k=1}^{n}c_kx_{k-1}
\in\mathfrak N_{n-1}.
\]
Repeating this argument yields $x_1\in\mathfrak N_0$. But this is impossible: every vector in $\mathfrak N_0=\bigvee\{x_0\}$ has second component $0$, since $x_0=(u_0,0)^T$, whereas
\[
x_1=
\begin{bmatrix}
u_1\\
v_1
\end{bmatrix}
\quad\text{with}\quad v_1=(B^*)^{-1}u_0\neq0.
\]
Therefore
\[
\mathfrak N_0\subsetneq \mathfrak N_1\subsetneq \mathfrak N_2\subsetneq\cdots.
\]

Finally, set
\[
\mathfrak M_n=\mathfrak N_n^\perp,\qquad n\ge0.
\]
Since each $\mathfrak N_n$ is invariant under $T^*$, each $\mathfrak M_n$ is invariant under $T$; and since the inclusions in the chain above are strict, so are those in
\[
\mathfrak M_0\supsetneq \mathfrak M_1\supsetneq \mathfrak M_2\supsetneq\cdots.
\]
This proves the theorem.
\end{proof}


\bibliographystyle{amsplain}

\end{document}